\documentclass[aap,MSNbibl,seceqn,citesort,dvips]{arximspdf}
\usepackage{mathbh}
\usepackage{graphicx}

\doi{10.1214/11-AAP832} 
\volume{23}
\issue{2}
\pubyear{2013}
\firstpage{492}
\lastpage{528}

\makeatletter

\newtheorem{theorem}{Theorem}[section]
\newtheorem{lemma}[theorem]{Lemma}
\newtheorem{claim}[theorem]{Claim}
\newtheorem{proposition}[theorem]{Proposition}

\newproclaim{definition}[theorem]{Definition}
\newproclaim{remark}{Remark}

\newcommand{\one}{\mathbh{1}}
\newcommand{\R}{\mathbb R}

\newcommand{\E}{\mathbb{E}}
\renewcommand{\P}{\mathbb{P}}

\renewcommand{\epsilon}{\varepsilon}

\newcommand{\cF}{\mathcal{F}}
\newcommand{\cC}{\mathcal{C}}
\newcommand{\cL}{\mathcal{L}}
\newcommand{\bin}{\operatorname{Bin}}
\newcommand{\tauc}{{\tau_c}}

\setattribute{abstract}    {skip} {20\p@}
\setattribute{keyword}     {skip} {8\p@}
\setattribute{frontmatter} {skip} {0\p@ plus 3\p@ minus 3\p@}

\makeatother

\begin{document}
\begin{frontmatter}

\title{Stochastic coalescence in logarithmic time}
\runtitle{Stochastic coalescence in logarithmic time}

\begin{aug}
\author[A]{\fnms{Po-Shen} \snm{Loh}\corref{}\ead[label=e1]{ploh@cmu.edu}}
\and
\author[B]{\fnms{Eyal} \snm{Lubetzky}\ead[label=e2]{eyal@microsoft.com}}
\runauthor{P.-S. Loh and E. Lubetzky}
\affiliation{Carnegie Mellon University and Microsoft Research}
\address[A]{Department of Mathematical Sciences\\
Carnegie Mellon University\\
Pittsburgh, Pennsylvania 15213\\
USA\\
\printead{e1}} 
\address[B]{Theory Group of Microsoft Research\\
One Microsoft Way\\
Redmond, Washington 98052\\
USA\\
\printead{e2}}
\end{aug}

\received{\smonth{3} \syear{2011}}
\revised{\smonth{11} \syear{2011}}

%
\begin{abstract}
The following distributed coalescence protocol was introduced by
Dahlia Malkhi in 2006 motivated by applications in social
networking. Initially there are $n$ agents wishing to coalesce into
one cluster via a decentralized stochastic process, where each round
is as follows: every cluster flips a fair coin to dictate whether it
is to issue or accept requests in this round. Issuing a request
amounts to contacting a cluster randomly chosen \textit{proportionally
to its size}. A cluster accepting requests is to select an
incoming one \textit{uniformly} (if there are such) and merge with
that cluster.
%
Empirical results by Fernandess and Malkhi suggested the protocol
concludes in $O(\log n)$ rounds with high probability, whereas
numerical estimates by Oded Schramm, based on an ingenious analytic
approximation, suggested that the coalescence time should be super-logarithmic.

Our contribution
is a rigorous study of the stochastic coalescence process with two
consequences. First, we confirm that the above process indeed
requires super-logarithmic time w.h.p., where the inefficient rounds
are due to oversized clusters that occasionally develop. Second, we
remedy this by showing that a simple modification produces an
essentially optimal distributed protocol; if clusters
favor their \textit{smallest} incoming merge request then the process
\textit{does} terminate in $O(\log n)$ rounds w.h.p., and simulations
show that the new protocol readily outperforms the original one.
Our upper bound hinges on a potential function involving the
logarithm of the number of clusters and the cluster-susceptibility,
carefully chosen to form a supermartingale. The analysis of the
lower bound builds upon the novel approach of Schramm which may find
additional applications: rather than seeking a single parameter that
controls the system behavior, instead one approximates the system by
the Laplace transform of the entire cluster-size distribution.
\end{abstract}

%

\end{frontmatter}

\section{Introduction}\label{intro}

The following stochastic distributed coalescence protocol was proposed
by Malkhi in 2006, motivated by applications in social networking and
the reliable formation of peer-to-peer networks (see~\cite{FM1}
for
more on these applications). The objective is to coalesce $n$
participating agents into a single hierarchal cluster reliably and
efficiently. To do so without relying on a centralized authority, the
protocol first identifies each agent as a cluster (a singleton), and
then proceeds in rounds as follows:
\begin{longlist}[(3)]
\item[(1)] Each cluster flips a fair coin to determine whether it will
be \textit{issuing} a merge-request or \textit{accepting} requests in the
upcoming round.
\item[(2)] Issuing a request amounts to selecting another cluster
randomly \textit{proportionally to its size}.
\item[(3)] Accepting requests amounts to choosing an incoming request
(if there are any) \textit{uniformly} at random and proceeding to merge
with that cluster.
\end{longlist}
In practice, each cluster is in fact a layered tree whose root is
entrusted with running the protocol, for example, each root decides
whether to issue or accept requests in a given round, etc. When
attempting to merge with another cluster, the root of cluster $\cC_i$
simply chooses a vertex $v$ uniformly out of $[n]$, which then
propagates the request to its root. This therefore corresponds to
choosing the cluster $\cC_j$ proportionally to $|\cC_j|$. This part of
the protocol is well-justified by the fact that agents within a cluster
typically have no information on the structure of other clusters in the system.

A second feature of the protocol is the symmetry between the roles of
issuing or accepting requests played by the clusters. Clearly, every
protocol enjoying this feature would have (roughly) at most half of its
clusters become acceptors in any given round, and as such could terminate
within $O(\log n)$ rounds. Furthermore, on an intuitive level, as long
as all clusters are of roughly the same size (as is the case
initially), there are few ``collisions'' (multiple clusters issuing a
request to the same cluster) each round and hence, the effect of a
round is similar to that of merging clusters according to a random
perfect matching. As such, one might expect that the protocol should
conclude with a roughly balanced binary tree in logarithmic time.

Indeed, empirical evidence by Fernandess and Malkhi~\cite{FM} showed
that this protocol seems highly efficient, typically taking
a logarithmic number of rounds to coalesce. However, rigorous
performance guarantees for the protocol were not available.

While there are numerous examples of stochastic processes that have
been successfully analyzed by means of identifying a single tractable
parameter that controls their behavior, here it appears that the entire
distribution of the cluster-sizes plays an essential role in the
behavior of the system. Demonstrating this is the following example:
suppose that the cluster $\cC_1$ has size $n-o(\sqrt{n})$ while all
others are singletons. In this case it is easy to see that with high
probability all of the merge-requests will be issued to $\cC_1$, who
will accept at most one of them\vadjust{\goodbreak} (we say an event holds \textit{with high
probability}, or w.h.p. for brevity, if its probability tends to $1$
as $n\to\infty$). Therefore, starting from this configuration,
coalescence will take at least $n^{1/2-o(1)}$ rounds w.h.p., a
polynomial slowdown.
Of course, this scenario is extremely unlikely to arise when starting
from $n$ individual agents, yet possibly other mildly unbalanced
configurations \textit{are} likely to occur and slow the process down.

In 2007, Schramm proposed a novel approach to the problem,
approximately reducing it to an analytic problem of determining the
asymptotics of a recursively defined family of real functions. Via this
approximation framework Schramm then gave numerical estimates
suggesting that the running time of the stochastic coalescence protocol
is w.h.p. super-logarithmic. Unfortunately, the analytical problem
itself seemed highly nontrivial and overall no bounds for the process
were known.

\subsection{New results}

In this work we study the stochastic coalescence process with two main
consequences. First, we provide a rigorous lower bound confirming that
this process w.h.p. requires a super-logarithmic number of rounds to
terminate. Second, we identify the vulnerability in the protocol,
namely the choice of which merge-request a cluster should approve.
While the original choice seems promising in order to maintain the
balance between clusters, it turns out that typical deviations in
cluster-sizes are likely to be amplified by this rule and lead to
irreparably unbalanced configurations. On the other hand, we show that
a simple modification of this rule to favor the smallest incoming
request is already enough to guarantee coalescence in $O(\log n)$
rounds w.h.p.
[Here and in what follows we let $f \lesssim g$ denote that $f = O(g)$
while $f \asymp g$ is short for $f \lesssim g \lesssim f$.]
%
%
\begin{theorem}\label{thm-1}
The uniform coalescence process $\mathcal{U}$ coalesces in $\tauc
(\mathcal{U}) \gtrsim\log n \cdot
\frac{\log\log n}{\log\log\log n}$ rounds w.h.p. Consider a
modified \textup{size-biased} process $\mathcal{S}$ where every accepting
cluster $\cC_i$ has the following rule:
\begin{itemize}
\item Ignore requests from clusters of size larger than $|\cC_i|$.
\item Among other requests (if any), select one issued by a cluster
$\cC_j$ of smallest size.
\end{itemize}
Then the coalescence time of the size-biased process satisfies $\tauc
(\mathcal{S}) \asymp\log n$ w.h.p.
\end{theorem}

Observe that the new protocol is easy to implement efficiently in
practice as each root can keep track of the size of its cluster and can thus
include it as part of the merge-request.

\subsection{Empirical results}
Our simulations show that the running time of the size-biased process
is approximately $5 \log_2 n$. Moreover, they further demonstrate that
the new size-biased process
empirically\vadjust{\goodbreak} performs substantially better than the uniform process
even for fairly small values of $n$, that is, the improvement appears
not only
asymptotically in the limit but already for ordinary input sizes.
These results are summarized in Figure~\ref{figruntimes}, where
the plot on the left clearly shows how the uniform process diverges from
%
%
\begin{figure}

\includegraphics{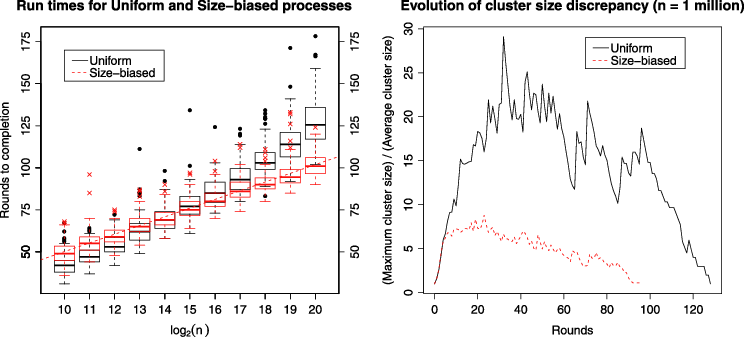}

\caption{The left plot compares the running times for
the two processes. Statistics are derived from 100 independent
runs of each process, for each $n \in\{1024, 2048, \ldots,
2^{20}\}$. The right plot tracks the ratio between the maximum
and average cluster-sizes, through a single run of each process,
for $n = 10^{6}$. There, the uniform process took 128 rounds,
while the size-biased process finished in 96.}
\label{figruntimes}
\end{figure}
the linear (in logarithmic scale) trend corresponding to the runtime of
the size-biased process. The right-most plot identifies the crux of the
matter; the uniform process
rapidly produces a highly skewed cluster-size distribution, which slows
it down considerably.

\subsection{Related work}

There is extensive literature on stochastic coalescence processes whose
various flavors fit the following scheme: the clusters act via a
continuous-time process where the coalescence rate of two clusters with
given masses $x,y$ (which can be either discrete or continuous) is
dictated up to re-scaling by a rate kernel $K$. A notable example of
this is Kingman's coalescent~\cite{Kingman}, which corresponds to the
kernel $K(x,y)=1$ and has been intensively studied in mathematical
population genetics (see, e.g.,~\cite{DurrettDNA} for more on Kingman's
coalescent and its applications in genetics). Other rate kernels that
have been thoroughly studied include the additive coalescent $K(x,y) =
x+y$ which corresponds to Aldous's continuum random
tree~\cite{Aldous91}, and the multiplicative coalescent $K(x,y)=xy$
that corresponds to Erd\H{o}s--R\'enyi random graphs~\cite{ER} (see the
books~\cite{BolBook,JLR}). For further information on these as well as
other coalescence processes, whose applications range from physics to
chemistry to biology, we refer the reader to the excellent survey of
Aldous~\cite{Aldous97}.

A major difference between the classical stochastic coalescence
processes mentioned above and those studied in this work is the
synchronous nature of the latter ones. Instead of individual merges
whose occurrences are governed by independent exponentials, here the
process is comprised of rounds where all clusters act simultaneously
and the outcome of a round (multiple disjoint merges) is a function of
these combined actions. This framework introduces delicate dependencies
between the clusters, and rather than having the coalescence rate of
two clusters be given by the rate kernel $K$ as a function of their
masses, here it is a function of the entire cluster distribution. For
instance, suppose nearly all of the mass is in one cluster $\cC_i$
(which thus attracts almost all merge requests); its coalescence rate
with a given cluster $\cC_j$ in the uniform coalescence process
$\mathcal{U}$ clearly depends on the total number of clusters at that
given moment, and similarly in the size-biased coalescence process
$\mathcal{S}$ it depends on the sizes of all other clusters, viewed as
competing with $\cC_j$ over this merge.
In face of these mentioned dependencies, the task of analyzing the
evolution of the clusters along the high-dimensional stochastic
processes $\mathcal{U}$ and $\mathcal{S}$ becomes highly nontrivial.


In terms of applications and related work in \textit{computer science}, the
processes studied here have similar flavor to those which arose in the 1980s,
most notably the \textit{random mate} algorithm introduced by Reif, and
used by
Gazit~\cite{Gazit} for parallel graph components and by Miller and
Reif~\cite{MillerReif} for parallel tree contraction.
However, as opposed to the setting of those algorithms, a key
difference here
is the fact that as the process evolves through time, each cluster is
oblivious to
the distribution of its peers at any given round (including the total
number of clusters for that matter).
Therefore, for instance, it is impossible for a cluster to sample from
the uniform distribution over the other clusters when issuing its merge request.

For another related line of works in \textit{computer science},
recall that the coalescence processes studied in this work organize $n$
agents in
a hierarchic tree, where each merged cluster reports to its acceptor
cluster. This is closely related to the rich and intensively studied
topic of randomized leader elections (see,
e.g., \mbox{\cite{CL,FMS,ORV,RZ,Zuckerman}}), where a computer network
comprised of $n$ processors attempts to single out a leader (in charge
of communication, etc.) by means of a distributed randomized process
generating the hierarchic tree. Finally, studying the dynamics of
randomly merging sets is also fundamental to
understanding the average-case performance of disjoint-set data
structures (see, e.g., the works of Bollob\'as and Simon~\cite{BS},
Knuth and Sch\"onhage~\cite{KS} and Yao~\cite{Yao}).
These structures, which are of fundamental importance in computer
science, store
collections of disjoint sets and support two operations; (i) taking the
union of a pair of sets and (ii) determining which set a particular
element is in (see, e.g.,~\cite{Galil} for a survey of these data
structures). The processes studied here precisely consider the
evolution of a
collection of disjoint sets under random merge operations and it is plausible
that the tools used here could contribute to advances in that area.

\subsection{Main techniques}

As we mentioned above, the main obstacle in the coalescence processes
studied here
is that since requests go to other clusters with probability
proportional to their size, the
largest clusters can create a bottleneck, absorbing all requests yet
each granting only one per round. An intuitive approach for analyzing
the size-biased process $\mathcal{S}$
would be to track a statistic that would warn against this scenario,
with the most obvious candidate
being the size of the largest cluster. However, simulations indicate that
this alone will be insufficient as the largest cluster does in fact grow
out of proportion in typical runs of the process. Nevertheless, the
distribution of large clusters
turns out to be sparse. The key idea is then to track a smoother
parameter involving the \textit{susceptibility}, which is essentially the
second moment of the cluster-size distribution.

To simplify notation, normalize the cluster-sizes $w_i$ to sum to 1 so
that the initial distribution consists of
$n$ clusters of size $\frac{1}{n}$ each. With this
normalization, the susceptibility $\chi_t$ is defined as $\sum_i
w_i^2$, the sum of
squares of cluster-sizes after the $t$th round. (We note in passing
that this
parameter has played a central role in the study of the
phase-transition in
percolation and random graphs; see, e.g.,~\cite{Grimmett,SW}.)
The proof that the size-biased protocol is optimal hinges on a
carefully chosen potential function $\Phi_t = \chi_t \kappa_t + C
\log
\kappa_t$, where $\kappa_t$ denotes the number of clusters after the
$t$th round and $C$ is an absolute constant chosen to turn $\Phi_t$
into a supermartingale. In Sections~\ref{secsize-E} and \ref
{secsize-whp} we will control the evolution of $\Phi_t$ and
prove our upper bound on the running time of the size-biased process.

The analysis of the uniform process $\mathcal{U}$ is delicate and
relies on rigorizing and analyzing the novel framework of Schramm \cite
{Schramm,Schramm2} for approximating the problem
by an analytic one. We believe this technique is of independent
interest and may find additional applications in the analysis of
high-dimensional stochastic processes. Instead of seeking a
single parameter to summarize the system behavior, one instead
measures the system using the Laplace transform of the entire
cluster-size distribution.
%
%
\begin{definition}\label{def-F-G}
For any integer $t\geq0$ let $\cF_t$ be the $\sigma$-algebra
generated by the first $t$ rounds of the process.
Conditioned on $\cF_t$, define the functions $F_t(s)$ and $G_t(s)$ on
the domain $\R$ as follows. Let $\kappa$ be the number of clusters
and let $w_1,\ldots, w_\kappa$ be the normalized cluster-sizes after
$t$ rounds. Set
%
%
\begin{equation}\label{defFG}
F_t(s) = \sum_{i=1}^{\kappa} \exp(-w_i s),\qquad
G_t(s) = \frac{1}{\kappa} F_t(\kappa s).
\end{equation}
\end{definition}

As we will further explain in Section~\ref{secoded}, the Laplace
transform $F_t$ simultaneously captures all the moments of the
cluster-size distribution, in a manner analogous to the moment
generating function of a random variable. This form is particularly
useful in our application as we will see in Section~\ref{secuniform}
that the specific evaluation $G_t ( \frac{1}{2} )$\vadjust{\goodbreak} governs
the expected coalescence rate. Furthermore, it turns out that it is
possible to estimate values of $F_t$ (and $G_t$) recursively. Although
the resulting recursion is nonstandard and highly complex, a somewhat
intricate analysis eventually produces a lower bound for the uniform
process.

\subsection{Organization}
The rest of this paper is organized as follows. In Section~\ref
{secoded} we describe Schramm's analytic approach for approximating
the uniform process~$\mathcal{U}$. Sections~\ref{secsize-E} and \ref
{secsize-whp} are devoted to the size-biased process $\mathcal{S}$. In
the former we prove that $\E[\tauc(\mathcal{S})]=O(\log n)$ and in the
latter we build on this proof together with additional ideas to show
that $\tauc(\mathcal{S})=O(\log n)$ w.h.p. The final section,
Section~\ref{secuniform}, builds upon Schramm's aforementioned
framework to produce a super-logarithmic lower bound for $\tauc
(\mathcal{U})$.

\section{Schramm's analytic approximation framework for the uniform
process}\label{secoded}

In this section we describe Schramm's analytic approach as it was
presented in~\cite{Schramm,Schramm2} for analyzing the uniform coalescence
process $\mathcal{U}$, as well as the numerical evidence that Schramm
obtained based on this approach suggesting that $\tauc(\mathcal{U})$ is
super-logarithmic. Throughout this section we write approximations
loosely as they were sketched by Schramm and postpone any arguments on
their validity (including concentration of random variables, etc.) to
Section~\ref{secuniform}, where we will turn elements from this
approach into a rigorous lower bound on $\tauc(\mathcal{U})$.

Let $\cF_t$ denote the $\sigma$-algebra generated by the first $t$
rounds of the coalescence process $\mathcal{U}$. The starting point of
Schramm's approach was to examine the following function conditioned on
$\cF_t$:
\[
F_t(s) = \sum_{i=1}^{\kappa_t} \exp(-w_i s),
\]
where $\kappa_t$ is the number of clusters after $t$ rounds and
$w_1,\ldots,w_{\kappa_t}$ denote the normalized cluster-sizes at that
time (see Definition~\ref{def-F-G}).
The benefit that one could gain from understanding the behavior of
$F_t(s)$ is obvious as $F_t(0)$ recovers the number of clusters at time $t$.

More interesting is the following observation of Schramm regarding the
role that $F_t(\kappa_t/2)$ plays in the evolution of the clusters.
Conditioned on $\cF_t$, the probability that the cluster $\cC_i$
receives a merge request from another cluster $\cC_j$ is $\frac12 w_i$
(the factor $\frac12$ accounts for the choice of $\cC_j$ to issue
rather than accept requests). Thus, the probability that $\cC_i$ will
receive \textit{any} incoming request in round $t+1$ and independently
decide to be an acceptor is
\[
\tfrac12 [1 - (1- w_i/2)^{\kappa_t-1}] \approx\tfrac12
[1-\exp( - w_i \kappa_t/2)].
\]
On this event, $\cC_i$ will account for one merge at time $t+1$, and
summing this over all clusters yields
\[
\E[\kappa_{t+1} \mid\cF_t ] \approx\kappa_t - \frac12 \sum
_{i=1}^{\kappa_t} [1-\exp( - w_i \kappa_t/2)] = \frac
12
[\kappa_t + F_t(\kappa_t/2)]\vadjust{\goodbreak}
\]
or equivalently, re-scaling $F_t(s)$ into $G_t(s) = (1/\kappa
_t)F_t(\kappa_t s)$ as in (\ref{defFG}),
%
%
\begin{equation}
\label{approxdrop}
\mathbb{E} [ \kappa_{t+1}/\kappa_t \mid\mathcal{F}_t ]
\approx
\frac{
1 + G_t(1/2)
}{2}.
\end{equation}
In order to have $\tauc(\mathcal{U}) \asymp\log n$ the number of clusters
would need to typically drop by at least a constant factor at each
round. This would
require the ratio in (\ref{approxdrop}) to be bounded away from
1, or equivalently, $G_t(\frac12)$ should be bounded away from 1.

Unfortunately, the evolution of the sequence $G_t(\frac12) = (1/\kappa
_t)F_t(\kappa_t/2)$ appears to be quite complex and there does not seem
to be a simple way to determine its limiting behavior. Nevertheless,
Schramm was able to write down an approximate recursion for the
expected value of $F_{t+1}$ in terms of multiple evaluations of $F_t$
by observing the following. On the above event that $\cC_i$ chooses to
accept the merge request of some other cluster $\cC_j$, by definition
of the process $\mathcal{U}$, the identity of the cluster $\cC_j$ is
uniformly distributed over all $\kappa_t-1$ clusters other than $\cC
_i$. Hence,
\begin{eqnarray*}
&&
\E[F_{t+1}(s)-F_t(s)\mid\cF_t] \\
&&\qquad\approx\sum_i \frac
12
(1-e^{-w_i \kappa_t/2} ) \frac1{\kappa_t}\sum_{j \neq i}
\bigl(e^{-(w_i + w_j)s} -e^{-w_i s} - e^{-w_j s}\bigr).
\end{eqnarray*}
Ignoring the fact that the last sum in the approximation skips the
diagonal terms $j=i$, one arrives at
a summation over all $1\leq i,j \leq\kappa_t$ of exponents similar to
those in the definition of $F_t$ with an argument of either $s$,
$\kappa
_t/2$, or $s + \kappa_t/2$, which, after rearranging, gives
%
\[
\E[ F_{t+1} (s) \mid\mathcal{F}_t ]
\approx\frac12 F_t(s+\kappa_t/2) + \frac{1}{2\kappa_t} F_t(s) [
F_t (s) + F_t(\kappa_t/2) - F_t(s+\kappa_t/2) ].
\]
To turn\vspace*{1pt} the above into an expression for $G_{t+1}(s)$ one needs to
evaluate $F_{t+1}(\kappa_{t+1} s)$ rather than $F_{t+1}(\kappa_t s)$,
to which\vspace*{2pt} end the approximation $\kappa_{t+1} \approx\frac12
[1+G_t(\frac12)] \kappa_t $ can be used based on (\ref{approxdrop}).
Additionally, for the starting point of the recursion, note that the
initial configuration of $w_i = 1/\kappa_0$ for all $1\leq i\leq
\kappa
_0$ has $G_0(s) = \exp(-s)$. Altogether, Schramm obtained the following
deterministic analytic recurrence,
whose behavior should (approximately) dictate the coalescence rate:
\[
\cases{
g_0 (s) = \exp(-s), \vspace*{2pt}\cr
\displaystyle g_{t+1}(s) =
{\frac{1}{2\alpha}\biggl[ g_t ( \alpha s )^2 - g_t \biggl( \alpha s +
\frac{1}{2} \biggr) g_t (\alpha s ) + g_t\biggl( \alpha s + \frac{1}{2}
\biggr) + g_t \biggl( \frac{1}{2} \biggr) g_t (\alpha s ) \biggr]},
\vspace*{2pt}\cr
\qquad\mbox{where $\displaystyle \alpha=
\frac12\biggl[1+g_t\biggl(\frac12\biggr)\biggr]$}.}
%
\]
In light of this, aside from the task of assessing how good of an
approximation the above defined functions $g_t$
provide for the random variables $G_t$ along the uniform coalescence
process $\mathcal{U}$,
the other key question is whether the sequence $g_t(\frac12)$ converges
to $1$ as $t\to\infty$, and if so, at what rate.

For the latter, as the complicated definition of $g_{t+1}$ attests,
analyzing the recursion of $g_t$ seems highly nontrivial. Moreover, a
naive evaluation of $g_t(\frac12)$ involves exponentially many terms,
making numerical simulations already challenging.
The computer-assisted numerical estimates performed by Schramm for the
above recursion, shown in Figure~\ref{figoded-numerics}, seemed to
suggest that indeed $g_t(\frac12)\to1$ (albeit very slowly), which
should lead to a super-logarithmic coalescence time for~$\mathcal{U}$.
However, no rigorous results were known for the limit of $g_t(\frac12)$
or its stochastic counterpart $G_t(\frac12)$.

%
%
\begin{figure}

\includegraphics{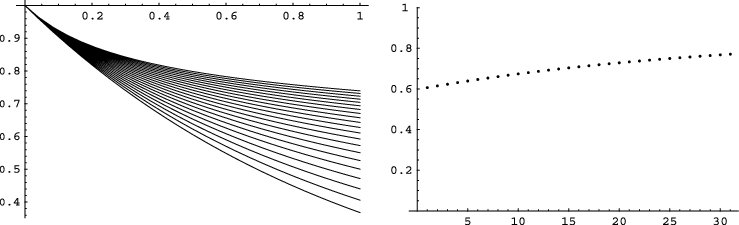}

\caption{Numerical estimations by Oded Schramm for the functions
$G_t(s)$ from his analytic approximation of the uniform coalescence
process. The left plot features $G_t(s)$ for $t=\{0,2,\ldots,40\}$ and
$s\in[0,1]$ and demonstrates how these increase with $t$. The right
plot focuses on $G_t(\frac12)$ and suggests that $G_t(\frac12)\to1$
and that in turn the coalescence rate should be super-logarithmic.}
\label{figoded-numerics}
\end{figure}

As we show in Section~\ref{secuniform}, in order to turn Schramm's
argument into a rigorous lower bound on $\tauc(\mathcal{U})$, we move
our attention away from the sought value of $G_t(\frac12)$ and focus
instead on $G_t(1)$. By manipulating Schramm's recursion for $G_t$ and
combining it with additional analytic arguments and appropriate
concentration inequalities, we show that as long as $\kappa_t$ is large
enough and $G_t(\frac12) < 1 - \delta$ for some fixed $\delta>0$,
then typically
$G_{t+1}(1) > G_t(1) + \epsilon$ for some $\epsilon(\delta) > 0$. Since
by definition $0 \leq G_t(1) \leq1$, this can be used to show that
ultimately $G_t(\frac12)\to1$ w.h.p., and a careful quantitative
version of this argument produces the rigorous lower bound on $\tauc
(\mathcal{U})$ stated in Theorem~\ref{thm-1}.

\section{Expected running time of the size-biased process}
\label{secsize-E}

The goal of this section is to prove that the expected time for the
size-biased process to complete has logarithmic order, as stated in
Proposition~\ref{propE-upper}.
Following a few simple observations on the process, we will prove this
proposition using two key lemmas, Lemmas~\ref{lemchi-kappa} and \ref
{lemno-A}, whose proofs will appear in Sections~\ref
{secprooflemchi-kappa} and~\ref{secprooflemno-A},
respectively. In Section~\ref{secsize-whp} we extend the proof of this
proposition using some additional ideas to establish that the
coalescence time is bounded by $O(\log n)$ w.h.p.\vadjust{\goodbreak}
%
%
\begin{proposition}
\label{propE-upper}
Let $\tauc=\tauc(\mathcal{S})$ denote the coalescence time of the
size-biased process $\mathcal{S}$. Then there
exists an absolute constant $C>0$ such that $\E_1 [\tauc] \leq C \log n$,
where $\E_1[\cdot]$ denotes expectation w.r.t. an initial cluster
distribution comprised of $n$ singletons.
\end{proposition}

Throughout Sections~\ref{secsize-E} and~\ref{secsize-whp} we refer
only to the size-biased process and use the following notation. Define
the filtration $\cF_t$ to be the $\sigma$-algebra generated by the
process up to and including the $t$th round.
Let $\kappa_t$ denote the number of clusters after the conclusion of
round $t$, noting that with these definitions we are interested in
bounding the expected value of the stopping time
%
%
\begin{equation}
\label{eq-tau-star-def}
\tauc= \min\{t\dvtx \kappa_t = 1\}.
\end{equation}
As mentioned in the \hyperref[intro]{Introduction}, we normalize the cluster-sizes so
that they sum to $1$. Finally, the susceptibility $\chi_t$ denotes the
sum of squares of the cluster-sizes at the end of round $t$.

Observe that by Cauchy--Schwarz, if $w_1,\ldots,w_{\kappa_t}$ are the
cluster-sizes at the end of round $t$ (and as such $\chi_t = \sum_{i}
w_i^2$) then we always have
%
%
\begin{equation}
\label{ineqchi-kappa}
\chi_t \kappa_t
\geq
\Biggl( \sum_{i=1}^{\kappa_t} w_i \Biggr)^2 = 1
\end{equation}
with equality iff all clusters have the same size. Indeed, the
susceptibility $\chi_t$ measures the variance of the cluster-size
distribution. When $\chi_t$ is smaller (closer to~$\kappa_t^{-1}$),
the distribution is more uniform. We further claim that
%
%
\begin{equation}
\label{ineqchi-grow-deterministic}
\chi_{t+1} \leq2\chi_t \qquad\mbox{for all $t$}.
\end{equation}
To see this, note that if a cluster of size $a$ merges with a
cluster of size $b$ the susceptibility increases by exactly
$ (a+b)^2 - (a^2 + b^2) = 2ab \leq a^2 + b^2 $.
Since each round only involves merges between disjoint pairs of
clusters, this immediately implies that the total additive increase in
susceptibility is bounded by the current sum of squares of the cluster
sizes, that is, the current susceptibility $\chi_t$.

Before commencing with the proof of Proposition~\ref{propE-upper}, we
present a trivial linear bound for the expected running time of the
coalescence process, which will later serve as the final step in our
proof. Here and in what follows, $\P_w$ and $\E_w$ denote probability
and expectation given the initial cluster distribution $w$. While the
estimate featured here appears to be quite crude when $w$ is uniform,
recall that in general $\tauc$ can in fact be linear in the initial
number of clusters w.h.p., for example, when $w$ is comprised of one
cluster of mass $1-1/\sqrt{n}$ and $\sqrt{n}$ other clusters of mass
$1/n$ each.
%
%
\begin{lemma}\label{lemtrivial}
Starting from $\kappa$ clusters with an arbitrary cluster distribu\-tion
$w=(w_1,\ldots,w_\kappa)$ we have $\E_w[\tauc] \leq8\kappa$.
Furthermore, $\P_w(\tauc> 16 \kappa) \leq e^{-\kappa/4}$.\vadjust{\goodbreak}
\end{lemma}
\begin{pf} Consider an arbitrary round in which at least $2$ clusters
still remain. We claim that the probability that there is at least
one merge in this round is at least~$\frac{1}{8}$. Indeed, let $\cC_1$
be a cluster of minimal size. The probability that it decides
to send a request is $\frac{1}{2}$, and since there are at least two
clusters and $\cC_1$ is the smallest one, the probability that this request
goes to some $\cC_j$ with $j \neq1$ is at least $\frac{1}{2}$.
Finally, the probability that $\cC_j$ is accepting requests is again
$\frac{1}{2}$. Conditioned on these events, $\cC_j$ will
definitely accept some request (possibly not the one from
$\cC_1$ as another cluster of the same size as $\cC_1$ may have sent
it a
request) leading to at least one merge, as claimed.

The process terminates when the total cumulative number of merges reaches
$\kappa- 1$. Therefore, the time of completion is stochastically
dominated by the sum of $\kappa-1$ geometric random variables with
success probability $\frac18$, and in particular $\E_w[\tauc] \leq
8(\kappa-1)$.

By the same reasoning, the total
number of merges that occurred in the first $t$ rounds clearly
stochastically dominates a binomial variable
$\bin(t,\frac18)$ as long as $t \leq\tauc$. Therefore,
\[
\P_w(\tauc> 16\kappa) \leq\P\bigl( \bin\bigl(16\kappa,\tfrac18\bigr)
\leq\kappa
-1\bigr)
\leq e^{-\kappa/4},
\]
where the last inequality used the well-known Chernoff bounds (see,
e.g.,~\cite{JLR}, Theorem 2.1).
\end{pf}

\subsection{\texorpdfstring{Proof of Proposition \protect\ref{propE-upper} via two key lemmas}
{Proof of Proposition 3.1 via two key lemmas}}
\label{secmain-upper-outline}
We next present the two main lemmas on which the proof of the
proposition hinges.
The key idea is to design a potential function comprised of two parts,
$\Phi_1,\Phi_2$, while identifying a certain event $A_t$ such that the
following holds: $\E[\Phi_1(t+1)-\Phi_1(t)\mid\cF_t,A_t]<c_1<0$ and
$\E[\Phi_2(t+1)-\Phi_2(t)\mid\cF_t] < c_2$, where $c_1,c_2$ are
absolute constants, and a similar statement holds conditioned on
$A_t^c$ when reversing the roles of $\Phi_1$ and $\Phi_2$. At this
point we will establish that an appropriate linear combination of $\Phi
_1,\Phi_2$ is a supermartingale, and the required bound on $\tauc$ will
follow from \textit{optional stopping}. Note that throughout the proof we
make no attempt to optimize the absolute constants involved.
The event $A_t$ of interest is defined as follows.
%
%
\begin{definition}\label{def-At}
Let $A_t$ be the event that the following two properties hold
after the $t$th round:
{\renewcommand\thelonglist{(\roman{longlist})}
\renewcommand\labellonglist{\thelonglist}
\begin{longlist}
\item\label{def-At-i} At least $\kappa_t/2$ clusters have size at most
$1/(600 \kappa_t)$.
\item\label{def-At-ii} The cluster-size
distribution satisfies $\sum_i w_i \one_{\{ w_i < 41/\kappa_t\}} < 4
\cdot10^{-5}$.
\end{longlist}}
\end{definition}

The intuition behind this definition is that property~\ref{def-At-i}
boosts the number of tiny clusters, thereby severely retarding the
growth of the largest clusters, which will tend to see incoming
requests from these tiny clusters. Property~\ref{def-At-ii} ensures
that most of the mass of the cluster-size distribution is on
relatively large clusters, of size at least 41 times the average.\vadjust{\goodbreak}

Examining the event $A_t$ will aid in tracking the variable $\chi_t
\kappa_t$, the normalized susceptibility [recall from (\ref
{ineqchi-kappa}) that this quantity is always at least $1$ and it
equals 1 whenever all clusters are of
the same size]. The next lemma, whose proof appears in Section \ref
{secprooflemchi-kappa}, estimates the expected change in this
quantity and most notably shows that it is at most $-\frac1{200}$ if we
condition on~$A_t$.

%
\begin{lemma}
\label{lemchi-kappa}
Let $\Phi_1(t)=\chi_t \kappa_t$ and suppose that at the end of the
$t$th round one has $\kappa_t \geq2$. Then
%
%
\begin{equation}
\label{eqchi-kappa}
\E[\Phi_1(t+1) - \Phi_1(t) \mid\cF_t] \leq5
\end{equation}
and furthermore,
%
%
\begin{equation}
\label{eqchi-kappa-A}
\E[\Phi_1(t+1) - \Phi_1(t) \mid\cF_t, A_t, \chi_t < 3
\cdot
10^{-7}]
\leq-\tfrac1{200}.
\end{equation}
\end{lemma}

Fortunately, when $A_t$ does not hold the behavior in the next round
can still be advantageous in the sense that in this case the number of
clusters tends to fall by at least a
constant fraction. This is established by the following lemma, whose
proof is postponed to Section~\ref{secprooflemno-A}.
%
%
\begin{lemma}
\label{lemno-A}
Let $\Phi_2(t) = \log\kappa_t$ and suppose that after the $t$th
round one has $\kappa_t \geq2$. Then
%
%
\begin{equation}
\label{eqkappa-drop-no-A}
\E[\Phi_2(t+1) - \Phi_2(t) \mid\cF_t, A^c_t] < - 2
\cdot10^{-7}.
\end{equation}
\end{lemma}

We are now in a position to derive Proposition~\ref{propE-upper} from
the above two lemmas.
\begin{pf*}{Proof of Proposition~\ref{propE-upper}}
Define the stopping time $\tau$ to be
\[
\tau= \min\{i\dvtx \chi_t \geq3 \cdot10^{-7}\}.
\]
Observe that the susceptibility is initially $1/n$, its value is $1$
once the process arrives at a single cluster (i.e., at time $\tauc$)
and until that point it is nondecreasing, hence, $\E\tau\leq\E\tauc
< \infty$ by Lemma~\ref{lemtrivial}.
Further define the random variable
\[
Z_t = \chi_t \kappa_t + 3 \cdot10^7 \log\kappa_t + \frac{t}{200}
.
\]
We claim that $(Z_{t \wedge\tau})$ is a supermartingale.
Indeed, consider $\E[Z_{t+1} \mid\cF_t, \tau> t]$ and note
that the fact that $\tau> t$ implies in particular that $\kappa_t
\geq2$
since in that case $\chi_t < 3\cdot10^{-7} < 1$.
\begin{itemize}
\item If $A_t$ holds then by (\ref{eqchi-kappa-A}) the
conditional expected change in $\chi_t \kappa_t$ is below $-\frac
{1}{200}$, while $\log\kappa_t$ can only decrease (as $\kappa_t$ is
nonincreasing), hence,
$\E[Z_{t+1} \mid\cF_t, A_t, \tau>t ] \leq Z_t$.
\item If $A_t$ does not hold, then by (\ref{eqchi-kappa}) the
conditional expected change in $\chi_t \kappa_t$ is at most $+5$ whereas
the conditional expected change in $\log\kappa_t$ is below $-2 \cdot
10^{-7}$ due to (\ref{eqkappa-drop-no-A}). By the scaling in the
definition of $Z_t$, these add up to give
$\E[Z_{t+1} \mid\cF_t, A^c_t, \tau> t] \leq Z_t - \frac{199}{200}$.
\end{itemize}
Altogether, $(Z_{t \wedge\tau})$ is indeed a supermartingale. As its
increments are bounded and the stopping time $\tau$ is integrable we
can apply the \textit{optional stopping theorem} (see, e.g.,
\cite{Durrett}, Chapter 5)
and get
%
%
\begin{equation}
\label{eqOST}
\E Z_{\tau} \leq Z_0 = \chi_0 \kappa_0 + 3 \cdot10^7 \log\kappa_0 =
O( \log n ).
\end{equation}
At the same time, by definition of $\tau$ we have $\chi_\tau\geq3
\cdot10^{-7}$ and so
%
%
\begin{equation}
\label{eqZT}
Z_\tau
=
\chi_\tau\kappa_\tau+ 3 \cdot10^7 \log\kappa_\tau+ \frac{\tau}{200}
\geq
3 \cdot10^{-7} ( \kappa_\tau+ \tau/8 ).
\end{equation}
Taking expectation in (\ref{eqZT}) and combining it with (\ref
{eqOST}) we find that
\[
\E[ \tau+ 8\kappa_\tau] \leq O( \log n ).
\]
Finally, conditioned on the cluster distribution at time $\tau$ we know
by Lem\-ma~\ref{lemtrivial}
that the expected number of additional rounds it takes the process to
conclude is at most $8 \kappa_\tau$, thus $\E[\tauc] \leq\E[\tau+
8\kappa_\tau]$. We can now conclude that $\E[\tauc] = O( \log n )$,
as required.
\end{pf*}

\subsection{\texorpdfstring{Proof of Lemma \protect\ref{lemchi-kappa}: Estimating the normalized susceptibility when $A_t$ holds}
{Proof of Lemma 3.4: Estimating the normalized susceptibility when $A_t$ holds}}\label{secprooflemchi-kappa}
The first step in controlling the product $\chi_t \kappa_t$ is to
quantify the
coalescence rate in terms of the susceptibility, as achieved by the
following claim.
%
%
\begin{claim}
\label{clmkappa-drop}
Suppose that at the end of the $t$th round one has $\kappa_t
\geq2$. Then
%
%
\begin{equation}
\label{eqkappa-drop}
\E[\kappa_{t+1} \mid\cF_t] \leq\kappa_t - (46\chi_t)^{-1}
\end{equation}
and furthermore,
\[
\P\bigl(\kappa_{t+1} < \kappa_t - (100 \chi_t)^{-1} \mid\cF_t,
\chi
_t < 3 \cdot10^{-7} \bigr) \geq1 - e^{-100}.
\]
\end{claim}
\begin{pf}
To simplify the notation let $\kappa= \kappa_t$, $\chi= \chi_t$ and
$\kappa' = \kappa_{t+1}$ throughout the proof of the claim. Further let
the clusters $\cC_i$ be indexed in
increasing order of their sizes and let $w_i = |\cC_i|$.

Recall that the number of merges in round $t+1$ is precisely the number
of clusters which decide to accept requests and then receive at least
one incoming request from a cluster of size no larger than itself.
Consider the probability of the latter event for a cluster $\cC_i$ with
$i > \lfloor\kappa/2 \rfloor$. Since the clusters are ordered by size
there are at least $\lfloor\kappa/2 \rfloor$ clusters of size at most
$w_i$ and each will send a request to $\cC_i$ independently with
probability $w_i/2$ (the factor of 2 is due to the probability of
issuing rather than receiving requests this round). The probability
that none of these clusters do so is thus at most $(1 -
w_i/2)^{\lfloor\kappa/2 \rfloor} \leq e^{-w_i \kappa/ 6}$ (where we
used the fact that $\lfloor\kappa /2\rfloor\geq\kappa/3$ for any
$\kappa\geq2$), and\vspace*{1pt} altogether the probability that
$\cC_i$ accepts a merge request from one of these clusters is at least
$\frac{1}{2} ( 1 - e^{-w_i \kappa/ 6} )$. Summing over these clusters
we conclude that
\[
\E[\kappa- \kappa' \mid\cF_t]
\geq
\sum_{i > \lfloor\kappa/2 \rfloor}
\frac{1}{2} ( 1 - e^{-w_i \kappa/ 6} )
\geq
\sum_{i=1}^\kappa
\frac{1}{4} ( 1 - e^{-w_i \kappa/ 6} ),
\]
where the last inequality follows from the fact that the summand is
increasing in $w_i$ and hence, the sum over the $\lceil\kappa/2\rceil$
largest clusters
should be at least as large as the sum over the $\lfloor\kappa
/2\rfloor
$ smallest ones.
Next, observe that by concavity, for all $0 \leq w_i \leq6\chi$ the
final summand is at least $w_i \cdot\frac{1}{4} ( 1 - e^{-\chi
\kappa} ) / (6\chi)$ which in turn is at least $ w_i \cdot\frac
{1}{4} ( 1 - e^{-1}
) / (6\chi)$ by (\ref{ineqchi-kappa}). As this last
expression always
exceeds $w_i/(38\chi)$ we get
%
%
\begin{equation}
\label{eq-kappa-kappa}
\E[\kappa- \kappa' \mid\cF_t]
\geq
\frac{1}{38\chi} \sum_{w_i \leq6\chi} w_i.
\end{equation}
We now aim to show that much of the overall mass is
spread on clusters of size at most $6\chi$. To this end recall that by
definition
$\chi= \sum w_i^2 $ while $\sum_i w_i=1$, hence, we can write $\chi=
\E Y$ where $Y$ is the random variable that accepts the value $w_i$
with probability $w_i$ for $i=1,\ldots,\kappa$. This gives that
\[
\sum_{w_i \leq6\chi} w_i = \P(Y \leq6 \E Y) > \frac56
\]
(with the final bound due to Markov's inequality)
and revisiting (\ref{eq-kappa-kappa}) we obtain that
\[
\E[\kappa- \kappa' \mid\cF_t] >
\frac{1}{38\chi} \cdot\frac{5}{6}
>
\frac{1}{46\chi},
\]
establishing inequality (\ref{eqkappa-drop}).

To complete the proof of the claim it suffices to show that the random
variable $X = \kappa- \kappa'$ is suitably concentrated, to which end
we use Talagrand's
inequality (see, e.g.,~\cite{MR}, Chapter 10).
In its following version we say that a function $f\dvtx\prod_i \Omega
_i\to\R
$ is $C$-Lipschitz if changing its argument $\omega$ in any single
coordinate changes $f(\omega)$ by at most $C$, and that $f$ is
$r$-certifiable if for every $s$ and $\omega$ with $f(\omega) \geq s$
there exists a subset $I$ of at most $rs$ coordinates such that every
$\omega'$ that agrees with $\omega$ on the coordinates indexed by
$I$ also has $f(\omega') \geq s$. In the context of a product space
$\Omega=\prod_i\Omega_i$ these definitions carry to the random variable
that $f$ corresponds to via the product measure.\looseness=1

%
\begin{theorem}[(Talagrand's inequality)]
If $X$ is a $C$-Lipschitz and $r$-certifiable random
variable on $\Omega= \prod_{i=1}^n \Omega_i$, then $ \P(|X -
\E X
| > t +\break 60C\sqrt{r\E X})
\leq
4\exp( -t^2/(8 C^2 r \E X))$ for any $0 \leq t \leq\E X$.
\end{theorem}

Observe that round $t+1$, conditioned on $\cF_t$,
is clearly a product space as the actions of the individual clusters are
independent. Formally, each cluster chooses either to accept requests
or to send a request to a random cluster. Changing the action of a
single cluster can only affect $X$, the number of merges in round \mbox{$t+1$},
by at most one merge and so $X$ is
1-Lipschitz. Also, if $X \geq s$ then one can identify $s$ clusters which
accepted merge requests from smaller clusters. By fixing the
decisions of the $2s$ clusters comprising these merges (the acceptors
together with
their corresponding requesters) we must have $X \geq s$
regardless of the other clusters' actions, as the $s$ acceptors
will accept (possibly different) merge-requests no matter what. Thus,
$X$ is also 2-certifiable.

Let $\mu= \E X$ and assume now that $\chi< 3 \cdot10^{-7}$.
By the first part of the proof [equation (\ref{eqkappa-drop})], it
then follows that
$\mu\geq(46\chi)^{-1} > 70\mbox{,}000$, in which case Talagrand's inequality gives
\[
\P\biggl( |X - \mu| > \frac{\mu}{6} + 60\sqrt{2\mu} \biggr)
\leq
4 \exp\bigl(-(\mu/6)^2/(16\mu)\bigr)
=
4 e^{-\mu/576}
<
e^{-100}.
\]
Also, note that our above bound $\mu> 70000 > 2 \cdot180^2$ implies
that
\[
60 \sqrt{2 \mu} < \mu/3,
\]
so in fact the probability of $X$ falling below $\mu- (
\frac{\mu}{6} + \frac{\mu}{3} )$ is at most $e^{-100}$. As
$\mu
\geq(46 \chi)^{-1}$ we conclude that $\kappa- \kappa' = X > (100
\chi)^{-1}$
with probability at least $1 - e^{-100}$, as required.
\end{pf}

As the above claim demonstrated the effect of the susceptibility on the
coalescence rate, we move to study the evolution of the susceptibility.
The critical advantage of the size-biased process is that large
clusters grow more slowly than small clusters. The intuition behind this
is that larger clusters tend to receive more requests, and since
clusters choose to accept their smallest incoming request, these clusters
typically have more choices to minimize over. It turns out that this
effect is enough to produce a useful quantitative bound on the
growth of the susceptibility.
%
%
\begin{claim}
\label{clmchi-grow}
Suppose that after the $t$th round $\kappa_t \geq2$. Then
%
%
\begin{equation}
\label{eqchi-grow}
\E[\chi_{t+1} \mid\cF_t] \leq\chi_t + \frac{5}{\kappa_t}.
\end{equation}
\end{claim}
\begin{pf}
Set $\kappa= \kappa_t$ and $\chi= \chi_t$. Let the clusters $\cC_i$
be indexed in increasing
order of their sizes and let $w_i = |\cC_i|$.
For each cluster $\cC_i$ let the random variable $X_i$
be the size of the smallest cluster that it receives a merge request
from, as long as that cluster is no larger than itself, and not
itself;\vadjust{\goodbreak} otherwise (the case where $\cC_i$ receives no merge requests
from another cluster of size less than or equal to its own) set $X_i = 0$.
Under these definitions we have
%
%
\begin{equation}
\label{eqchi}
\E[\chi_{t+1} \mid\cF_t]
= \chi+ \sum_{i=1}^\kappa w_i \E[X_i],
\end{equation}
since\vspace*{1pt} each $\cC_i$ is an acceptor with probability $\frac12$ and if it
indeed accepts a request from a cluster of size $X_i$ then the
susceptibility will increase by exactly $(w_i+X_i)^2 - (w_i^2 + X_i^2)
= 2w_i X_i$.

Next, note that since we ordered the clusters by increasing order of
size, each of the
first $\lfloor\kappa/2 \rfloor$ clusters has size at
most $2/\kappa$ (otherwise the last $\lceil\kappa/2\rceil$ clusters
would combine to a total mass larger than $1$). We will use this fact
to bound \mbox{$\E[X_i \mid\cF_t]$} by
considering two situations:
\begin{longlist}[(2)]
\item[(1)] If $\cC_i$ receives an
incoming request from at least one of the first $\lfloor
\kappa/2 \rfloor$ clusters (including itself), then $X_i
\leq2/\kappa$ by the above argument. The probability of this is
precisely
$
1 - (
1 - \frac{w_i}{2}
)^{\lfloor\kappa/2 \rfloor}
$
as each of the first $\lfloor\kappa/2 \rfloor$
clusters $\cC_j$ independently sends a request to $\cC_i$ with probability
$w_i/2$ (with the factor of $2$ due to the decision of $\cC_j$ whether
or not to issue requests).
\item[(2)] If $\cC_i$ gets no requests from the first $\lfloor
\kappa/2\rfloor$ clusters, then use the trivial
bound $X_i \leq w_i$.
\end{longlist}
Combining the two cases we deduce that
%
%
\begin{equation}
\label{eqXi}
\E X_i \leq
\biggl(
1 -
\biggl(
1 - \frac{w_i}{2}
\biggr)^{\lfloor\kappa/2 \rfloor}
\biggr)
\frac{2}{\kappa}
+
\biggl(
1 - \frac{w_i}{2}
\biggr)^{\lfloor\kappa/2 \rfloor}
w_i.
\end{equation}
We claim that $\E X_i $ is in fact always at most $5/\kappa$.
To see this, first note that if $w_i \leq2/\kappa$ then this
immediately holds, for example, since $X_i \leq w_i$. Consider therefore
the case where $w_i > 2/\kappa$. Since (\ref{eqXi}) is a weighted
average of $2/\kappa$ and
$w_i>2/\kappa$, it increases whenever the weight on
$w_i$ is increased. As
\[
\biggl(
1 - \frac{w_i}{2}
\biggr)^{\lfloor\kappa/2 \rfloor}
\leq
e^{-(w_i/2) \lfloor\kappa/2 \rfloor}
\leq
e^{-w_i \kappa/6},
\]
we have that, in this case,
\[
\E X_i \leq( 1 - e^{ -w_i \kappa/6 } ) \frac{2}{\kappa}
+ e^{ -w_i \kappa/6 }
w_i \leq
\frac{1}{\kappa}
(
2
+
w_i \kappa e^{ -w_i \kappa/6 }
).
\]
One can easily verify that the function $f(x)=x e^{-x/6}$ satisfies
$f(x) \leq3$ for all~$x$, hence, we conclude that $\E X_i \leq
5/\kappa$
in all cases, as claimed. Plugging this into (\ref{eqchi})
we obtain that
\[
\E[\chi_{t+1} \mid\cF_t]
\leq
\chi+ \frac{5}{\kappa} \sum_{i=1}^\kappa w_i
=
\chi+ \frac{5}{\kappa}
\]
as required.\vadjust{\goodbreak}
\end{pf}

While the last claim allows us to limit the growth of the
susceptibility, this bound is unfortunately too weak in general.
For instance, when used in tandem with Claim~\ref{clmkappa-drop}, it
results in the susceptibility growing out of control, while the number of
clusters decreases slower and slower. Crucially, however, conditioned on
the event $A_t$ (as given in Definition~\ref{def-At}) we can refine
these bounds to show that the growth of $\chi_{t+1}$ slows down
dramatically, as the following claim establishes.
%
%
\begin{claim}
\label{clmchi-grow-A}
Suppose that at the end of the $t$th round $\kappa_t \geq2$. Then
%
%
\begin{equation}
\label{eqchi-grow-A}
\E[\chi_{t+1} \mid\cF_t, A_t] \leq\chi_t + (201 \kappa_t)^{-1}
.
\end{equation}
\end{claim}
\begin{pf}
Let $\kappa= \kappa_t$ and $\chi= \chi_t$, and define the random
variables $X_i$
as in the proof of Claim~\ref{clmchi-grow}. By the same reasoning used
to deduce
inequality (\ref{eqXi}), only now using property~\ref{def-At-i} of
$A_t$ according to which each of the smallest $\lceil\kappa/2\rceil$
clusters has size at most
$1/(600 \kappa_t)$, we have
%
%
\begin{equation}
\label{eqXi-A}
\E X_i \leq \biggl( 1 - \biggl( 1 - \frac{w_i}{2}
\biggr)^{\lceil\kappa/2\rceil} \biggr) \frac{1}{600\kappa} + \biggl( 1
- \frac{w_i}{2} \biggr)^{\lceil\kappa/2\rceil} w_i.
\end{equation}
Recall that equation (\ref{eqchi}) established that $\E[\chi_{t+1}
\mid\cF_t] = \chi
+ \sum_{i=1}^\kappa w_i \E X_i$. This time we will need to bound this
sum more delicately by splitting it into two parts based on
whether or not $w_i < 41/\kappa$. In the case $w_i < 41/\kappa$ we can
use the trivial bound $X_i \leq w_i $ to arrive at
\[
\sum_{i} w_i \one_{\{w_i < 41/\kappa\}} \E X_i <
\sum_{i} w_i \one_{\{w_i < 41/\kappa\}} \frac{41}{\kappa}
<
4 \cdot10^{-5} \cdot\frac{41}{\kappa},
\]
where the last inequality is by property~\ref{def-At-ii} of $A_t$.
For the second part of the summation we use the same weighted mean
argument from the proof of Claim~\ref{clmchi-grow} to deduce that when
$w_i > (600\kappa)^{-1}$, the right-hand side of (\ref{eqXi-A})
increases with the weight on $w_i$, which in turn is at most
$(1 - \frac{w_i}2)^{\lceil\kappa/2\rceil} \leq\exp(-w_i
\kappa
/4)$. In particular, in case $w_i \geq41/\kappa$, we have
\begin{eqnarray*}
\E X_i &\leq&
( 1 - e^{ -w_i \kappa/4 } ) \frac{1}{600\kappa}
+ e^{ -w_i \kappa/4 } w_i
\leq
\frac{1}{\kappa} \biggl( \frac{1}{600} + w_i \kappa e^{ -w_i
\kappa/4
}\biggr)\\
&\leq&
\frac{1}{\kappa}
\biggl(\frac{1}{600} + 41 e^{ -41/4 }\biggr)
\end{eqnarray*}
(here we used the fact that the function $xe^{-x/4}$ is decreasing for
$x \geq41$). Combining our bounds,
\[
\sum_{i=1}^\kappa w_i \E X_i
\leq
\frac{1}{\kappa}
\biggl(
4 \cdot10^{-5} \cdot41
+
\sum_{i} w_i \one_{\{w_i \geq41/\kappa\}} \biggl( \frac{1}{600} + 41
e^{ -41/4 } \biggr)
\biggr) <
\frac{1}{201 \kappa}
\]
since $\sum_{i} w_i = 1$. Together with
(\ref{eqchi}), the proof is complete.
\end{pf}

Combining the bound on $\kappa_{t+1}$ in Claim~\ref{clmkappa-drop}
with the bounds on $\chi_{t+1}$ from Claims~\ref{clmchi-grow}
and \ref
{clmchi-grow-A} will now result in the statement of Lemma~\ref{lemchi-kappa}.
\begin{pf*}{Proof of Lemma~\ref{lemchi-kappa}}
For convenience let $\kappa= \kappa_t$ and $\chi= \chi_t$, as well as
$\kappa' = \kappa_{t+1}$ and $\chi' = \chi_{t+1}$.
The first statement of the lemma is an immediate consequence of
Claim~\ref{clmchi-grow} since $\kappa' \leq\kappa$ and so
\[
\E[\chi' \kappa' \mid\cF_t]
\leq
\kappa\E[\chi' \mid\cF_t]
\leq
\kappa\biggl(
\chi+ \frac{5}{\kappa}
\biggr)
=
\chi\kappa+ 5.
\]
For the second statement, since we can break down $\chi'\kappa'$ into
\begin{eqnarray*}
\chi'\kappa' &=& \chi' \biggl( \kappa- \frac{1}{100\chi} \biggr) + \chi'
\biggl( \kappa' - \kappa+ \frac{1}{100\chi} \biggr) \one_{\{\kappa'
\geq\kappa- {1}/({100\chi})\}}\\
&&{} + \chi' \biggl( \kappa' - \kappa+
\frac{1}{100\chi} \biggr) \one_{\{\kappa' < \kappa-
{1}/({100\chi})\}},
\end{eqnarray*}
noticing that the last expression in the right-hand side is at most
$0$, and recalling that $0 < \chi\leq\chi' \leq2\chi$ [due to
(\ref{ineqchi-grow-deterministic})] and $1 \leq\kappa' \leq
\kappa$, we now obtain that $\E[\chi' \kappa' \mid\cF_t,
A_t
,\chi< 3 \cdot10^{-7}]$ is at most
\begin{eqnarray*}
&&\E\biggl[\chi' \biggl( \kappa- \frac{1}{100\chi} \biggr) \Bigm|
\cF
_t, A_t,\chi< 3 \cdot10^{-7}\biggr]\\
&&\quad{}+ \E\biggl[2\chi\cdot\frac{1}{100\chi} \one_{\{\kappa' \geq
\kappa
-1/({100\chi})\}} \Bigm| \cF_t, A_t,\chi< 3 \cdot
10^{-7}
\biggr] \\
&&\qquad= \biggl( \kappa- \frac{1}{100\chi} \biggr) \E[\chi'
\mid
\cF_t, A_t,\chi< 3 \cdot10^{-7}]\\
&&\qquad\quad{}
+ \frac1{50}\P\biggl( \kappa' \geq\kappa-\frac1{100\chi}
\Bigm| \cF
_t, A_t,\chi< 3 \cdot10^{-7}\biggr).
\end{eqnarray*}
Applying Claims~\ref{clmkappa-drop} and~\ref{clmchi-grow-A} now gives
\begin{eqnarray*}
\E[\chi' \kappa' \mid\cF_t, A_t,\chi< 3 \cdot
10^{-7}]
&\leq&\biggl( \kappa- \frac{1}{100\chi} \biggr) \biggl( \chi+
\frac
{1}{201\kappa} \biggr)
+ \frac{1}{50} e^{-100} \\
&<& \chi\kappa- \frac{1}{100} + \frac{1}{201} + \frac{1}{50} e^{-100}\\
&<& \chi\kappa- \frac{1}{200}
\end{eqnarray*}
and the proof is complete.
\end{pf*}

\subsection{\texorpdfstring{Proof of Lemma \protect\ref{lemno-A}: Estimating the number of components when $A_t$ fails}
{Proof of Lemma 3.5: Estimating the number of components when $A_t$ fails}}\label{secprooflemno-A}

We wish to show that whenever either one of the two properties
specified in $A_t$ does not hold, the expected number of clusters drops
by a constant factor.\vadjust{\goodbreak}

Suppose that property~\ref{def-At-i} of $A_t$ fails. In this case a
constant fraction of the clusters have size which is at least a
constant fraction of the average size $1/\kappa_t$. We will show that
each such cluster receives an incoming request (from another cluster of
no larger size) in the next round with a probability that is uniformly
bounded from below. Consequently, we will be able to conclude that the
number of clusters shrinks by at least a constant factor in expectation.
%
%
\begin{claim}
\label{clmno-Ai}
Suppose that at the end of the $t$th round $\kappa_t \geq2$
and property~\ref{def-At-i} of $A_t$ does not hold, that is, more
than $\kappa_t/2$ clusters have size greater than
$(600 \kappa_t)^{-1}$. Then
%
%
\begin{equation}
\label{eqkappa-drop-no-Ai}
\E[\kappa_{t+1} \mid\cF_t]
\leq
(1 - 5 \cdot10^{-5}) \kappa_t.
\end{equation}
\end{claim}
\begin{pf}
Let $\kappa= \kappa_t$ and $\kappa' = \kappa_{t+1}$ and as usual,
order the clusters by increasing order of size. Consider an arbitrary cluster
$\cC_i$ which is one of the last $\lceil\kappa/2 \rceil$ clusters, and
let $w_i$ denote its size. If $\cC_i$ opts to
accept requests in this round (with probability $\frac12$) and any of the
first $\lfloor\kappa/2\rfloor$ clusters sends it a
request, it will contribute a merge in this round. This occurs
with probability
\[
\frac{1}{2} \biggl( 1 - \biggl( 1 - \frac{w_i}{2} \biggr)^{\lfloor
\kappa
/2 \rfloor} \biggr)
\geq
\frac{1}{2}
(
1 - e^{ -w_i \kappa/6 }
) > \frac{1}{2} (1 -
e^{-1/3600}) > 10^{-4},
\]
where we used our assumption that $w_i \geq(600 \kappa)^{-1}$. Thus,
the probability
that $\cC_i$ contributes to a merge is at least $10^{-4}$. We conclude
that the expected number of merges in this round is at least $10^{-4}
\lceil\kappa/2 \rceil$, from which the desired
result follows.
\end{pf}

Now suppose that property~\ref{def-At-ii} of $A_t$ fails. Here at
least a
constant proportion of the mass of the cluster-size distribution falls
on clusters with size at most a constant multiple of the average
size. Such clusters behave nicely
as in this window the relation between the cluster-size and the
typical number of incoming requests can be bounded by a linear
function. Again, this will result in a constant proportion of
clusters merging in the next round in expectation.
%
%
\begin{claim}
\label{clmno-Aii}
Suppose that at the end of the $t$th round $\kappa_t \geq2$
and property~\ref{def-At-ii} of $A_t$ does not hold, that is,
$\sum_i w_i \one_{\{ w_i < 41/\kappa_t\}} \geq4 \cdot10^{-5}$,
where $w_i$ denotes the size of $\cC_i$. Then
%
%
\begin{equation}
\label{eqkappa-drop-no-Aii}
\E[\kappa_{t+1} \mid\cF_t]
\leq
(1 - 2 \cdot10^{-7}) \kappa_t.
\end{equation}
\end{claim}
\begin{pf}
Let $\kappa= \kappa_t$ and $\kappa' = \kappa_{t+1}$. Order the
clusters by size and let $r$ be the number of clusters
which are smaller than $41/\kappa$. Since clearly at most
$\kappa/41$ clusters can have size at least\vadjust{\goodbreak}
$41/\kappa$, we have $r \geq\lceil
\frac{40}{41} \kappa\rceil$. Notice that since $\kappa\geq2$,
this implies
that in particular $\lfloor r/2\rfloor\geq\kappa/3$.
By the same arguments as before,
each cluster $\cC_i$ with $\lfloor r/2 \rfloor< i \leq
r$ will accept a merge request from a smaller cluster with probability
at least
\[
\frac{1}{2}
\biggl(
1 -
\biggl( 1 - \frac{w_i}{2} \biggr)^{\lfloor r/2 \rfloor}
\biggr)
\geq\frac{1}{2}\bigl( 1 - e^{ -(w_i/2) \lfloor r/2 \rfloor} \bigr)
\geq\frac{1}{2}( 1 - e^{ -w_i \kappa/6} ).
\]
Since we are concentrating our attention on
the clusters of size $w_i < 41/\kappa$, concavity implies that
the last expression is actually at least
\[
\frac{1}{2}
( 1 - e^{ -41/6} )
\frac{w_i}{41/\kappa}
>
\frac{w_i \kappa}{100}.
\]
We conclude that the expected number of merges in
this round is at least
\[
\sum_{i = \lfloor r/2\rfloor+ 1}^r
\frac{w_i \kappa}{100}
\geq
\frac{\kappa}{100} \cdot\frac{1}{2} \sum_{i=1}^r w_i
\geq\frac{\kappa}{100} \cdot\frac{1}{2} \cdot4 \cdot10^{-5}
=
2 \cdot10^{-7} \kappa,
\]
where we used the fact that the $w_i$'s are sorted in increasing
order to relate the sum over the cluster indices $\lfloor r/2\rfloor
+1,\ldots,r $ to the one over
the first $r$ clusters. This gives the
desired result.
\end{pf}
\begin{pf*}{Proof of Lemma~\ref{lemno-A}}
The proof readily follows from the combination of Claims \ref
{clmno-Ai} and~\ref{clmno-Aii}. Indeed, these claims establish that
whenever the event $A_t$ fails we have
\[
\E[\kappa_{t+1} \mid\cF_t, A_t^c] \leq
(1 - 2 \cdot10^{-7}) \kappa_t.
\]
Therefore, by the concavity of the logarithm, Jensen's inequality
implies that
\begin{eqnarray*}
\E[\log\kappa_{t+1} \mid\cF_t, A_t^c]
&\leq&
\log\E[\kappa_{t+1} \mid\cF_t, A_t^c]
\leq
\log\kappa_t + \log(1 - 2 \cdot10^{-7})\\
&<&
\log\kappa_t - 2 \cdot10^{-7}
\end{eqnarray*}
as required.
\end{pf*}

\section{Optimal upper bound for size-biased process}
\label{secsize-whp}

We now prove the upper bound in Theorem~\ref{thm-1} by building upon
the ideas of the previous section. Recall that in the proof of
Proposition~\ref{propE-upper} we defined the sequence
\[
Z_t
=
\chi_t \kappa_t + M \log\kappa_t + \frac{t}{200}
\qquad\mbox
{where $M
= 3 \cdot10^7$},
\]
established that it was a supermartingale and derived the required
result from optional
stopping. That approach was only enough to produce a bound on $\E
[\tauc
]$, the
expected completion time. For the stronger result on the typical value
of $\tauc$ we will analyze $(Z_t)$ more delicately. Namely, we estimate
its increments in $\cL^2$\vadjust{\goodbreak} to qualify an application of an appropriate
Bernstein--Kolmogorov large-deviation inequality for supermartingales
due to Freedman~\cite{Freedman}.

An important element in our proof is the modification of the above
given variable $Z_t$ into an overestimate $Y_t$ which allows
far better control over the increments in $\cL^2$. This is defined as
%
%
\begin{eqnarray}\label{eq-Y0-def}
Y_0 &=& Z_0 = \chi_0 \kappa_0 + M \log\kappa_0 = 1 + M \log
n,\nonumber\\[-8pt]\\[-8pt]
Y_{t+1} &=& \cases{
\displaystyle Y_t + (
\Xi_{t+1} \wedge\log^{2/3} n
)
+ M \log\frac{\kappa_{t+1}}{\kappa_{t}}
+ \frac1{200}, &\quad if $\tauc> t$,\vspace*{2pt}\cr
Y_t, &\quad if $\tauc\leq t$,}
\nonumber
\end{eqnarray}
where
\[
\Xi_{t+1}
=
\chi_{t+1} \biggl(
\kappa_{t+1} \vee\biggl( \kappa_{t} - \frac1{\chi_{t}} \biggr)
\biggr)
- \chi_{t}\kappa_{t}.
\]
The purpose of the $( \kappa_t - \frac{1}{\chi_t} )$ term is
to limit the potential decrease from negative $\Xi$. In this section,
we will need two-sided estimates (in addition to one-sided bounds such as
those used in the previous section) due to the fact that we must control
the $\cL^2$ increments.

It is clear that $Y_{t+1}-Y_{t} \geq Z_{t+1}-Z_{t}$ as long as $t <
\tauc$ and $\Xi_{t+1} \leq\log^{2/3} n$. Therefore, setting
\[
\bar{\tau} = \min\{t\dvtx\Xi_{t+1} > \log^{2/3} n\},
\]
it follows that
%
%
\begin{equation}
\label{eqYZ}
Y_{t} \geq Z_t \qquad\mbox{for all $t \leq\tauc\wedge\bar{\tau}$.}
\end{equation}
In what follows we will establish a large deviation estimate for
$(Y_t)$, then use this overestimate for $Z_t$ to show that w.h.p.
$\tauc= O(\log n)$. We thus focus our attention on the sequence $(Y_t)$.
%
%
\begin{lemma}
\label{lemY-supermart}
The sequence $(Y_t)$ is a supermartingale.
\end{lemma}
\begin{pf}
Since by definition $Y_t = Y_{t \wedge\tauc}$, it suffices to
consider the times $t < \tauc$. As we clearly have $(\kappa_{t+1}
\vee(\kappa_{t} - \frac1{\chi_{t}})) \leq\kappa_t$ and
Claim~\ref{clmchi-grow} established that $\E[ \chi_{t+1} \mid
\cF_t ] \leq\chi_t + \frac{5}{\kappa_t}$, we can deduce that
%
%
\begin{equation}
\label{eqXi-5}
\E[ \Xi_{t+1} \mid\cF_t ] \leq5.
\end{equation}
Combined with Lemma~\ref{lemno-A} as in the proof of
Proposition~\ref{propE-upper}, it then follows that
\[
\E[ Y_{t+1} \mid\cF_t,A_t^c] \leq0.
\]
We turn to consider $ \E[ Y_{t+1} \mid\cF_t,A_t] $. Since
$\kappa_{t+1} \leq\kappa_t$ holds for all $t$, it suffices to show
that
\[
\E[ \Xi_{t+1} \mid\cF_t,A_t ] \leq-\frac1{200}.\vadjust{\goodbreak}
\]
Indeed, as in the proof of Lemma~\ref{lemchi-kappa}, we write
\begin{eqnarray*}
\Xi_{t+1} &\leq&\chi_{t+1}\biggl(\kappa_t - \frac{1}{100\chi
_t}\biggr)\\
&&{}+ \chi_{t+1}\biggl[ \biggl( \kappa_{t+1} \vee\biggl( \kappa_t - \frac
{1}{\chi_t} \biggr) \biggr)
-\kappa_t +
\frac{1}{100\chi_t}\biggr]\one_{\{\kappa_{t+1} \geq
\kappa_t-1/({100\chi_t})\}}
- \chi_t\kappa_t \\
&\leq&\chi_{t+1}\biggl(\kappa_t - \frac{1}{100\chi_t}\biggr) +
2\chi_t\cdot\frac1{100\chi_t} \one_{\{\kappa_{t+1} \geq
\kappa_t-1/({100\chi_t})\}}- \chi_t\kappa_t,
\end{eqnarray*}
which as stated before gives rise to
\[
\E[\Xi_{t+1} \mid\cF_t, A_t] < - \tfrac{1}{100} +
\tfrac{1}{201} + \tfrac{1}{50}
e^{-100} < - \tfrac{1}{200},
\]
and we conclude that $(Y_t)$ is indeed a supermartingale, as required.
\end{pf}
%
%
\begin{lemma} The increments of the supermartingale $(Y_t)$ are
uniformly bounded in $\cL^2$. Namely, for every $t$ we have $\E
[(Y_{t+1}-Y_t)^2 \mid\cF_t] < 2 M^2$ where $M = 3\cdot10^7$.
\end{lemma}
\begin{pf}
First observe that
%
%
\begin{equation}
\label{eqY-CS}
(Y_{t+1}-Y_t)^2
\leq
3 (\Xi_{t+1})^2
+ 3 \biggl( M \log\frac{\kappa_{t+1}}{\kappa_t} \biggr)^2
+ 3 \biggl( \frac{1}{200} \biggr)^2.
\end{equation}
Since $\frac{1}{2} \kappa_t \leq\kappa_{t+1} \leq\kappa_t$, we
have $- M \log2 \leq M \log\frac{\kappa_{t+1}}{\kappa_t} \leq0$,
hence, the last two expressions above sum to, at most, $\frac32 M^2$
(with room to spare)
and it remains to bound $\E[(\Xi_{t+1})^2 \mid\cF_t] =O(1)$ for a
suitably small implicit constant.

Observe that when $\Xi_{t+1} \geq0$ we must have $ |\Xi_{t+1}|
\leq\chi_{t+1} \kappa_t - \chi_t\kappa_t$ since $(\kappa_{t+1}
\vee(\kappa_{t} - \frac1{\chi_{t}})) \leq
\kappa_t$. Conversely, if $\Xi_{t+1} \leq0$ then necessarily
$|\Xi_{t+1}| \leq\chi_t \kappa_t - \chi_{t+1} (\kappa_{t} -
\frac1{\chi_{t}}) \leq1$, with the\vspace*{1pt} last inequality due to the
fact that $\kappa_t \geq1/\chi_t$ and $\chi_{t+1} \geq
\chi_t$. Combining the cases we deduce that, in particular,
\[
|\Xi_{t+1}| \leq\kappa_t (\chi_{t+1}-\chi_t) + 1.
\]
By Claim~\ref{clmchi-grow} we have $\E[ \chi_{t+1}-\chi_t\mid\cF_t]
\leq5/\kappa_t$, hence, we get
%
%
\begin{eqnarray}\label{eq-Xi-sq}
\E[ (\Xi_{t+1})^2 \mid\cF_t ]
&\leq&\kappa_t^2 \E[ (\chi_{t+1}-\chi_t)^2 \mid\cF_t ] + 1 + 2
\kappa
_t (5/\kappa_t)\nonumber\\[-8pt]\\[-8pt]
&\leq&\kappa_t^2 \E[ (\chi_{t+1}-\chi_t)^2 \mid\cF_t ] + 11.\nonumber
\end{eqnarray}
It remains to show that $\E[ (\chi_{t+1}-\chi_t)^2 \mid\cF_t
]=O(1/\kappa_t^2)$. To do so, let $w_1, \ldots,\break w_{\kappa_t}$ be the
cluster-sizes after the $t$th round and recall
that by (\ref{eqchi}) and the arguments following it we have
\[
\E[(\chi_{t+1}-\chi_t)^2 \mid\cF_t]
=
\E\Biggl[ \Biggl(\sum_{i=1}^{\kappa_t} 2 w_i X_i I_i\Biggr)^2 \Biggr],
\]
where each $X_i$ is a nonnegative random variable satisfying
$\E X_i \leq5/\kappa_t$ (marking the size of another cluster of no
larger size that issued a request to $\cC_i$ or 0 if there was no such
cluster) and each $I_i$ is a Bernoulli($\frac12$) variable independent
of $X_i$ (indicating whether or not $\cC_i$ chose to accept requests).
Since $\sum w_i = 1$, it follows from convexity that
\[
\Biggl(\sum_{i=1}^{\kappa_t} w_i X_i I_i \Biggr)^2
\leq
\sum_{i=1}^{\kappa_t} w_i X_i^2 I_i,
\]
hence, taking expectation while recalling that $I_i$ and $X_i$ are independent,
\[
\E[(\chi_{t+1}-\chi_t)^2 \mid\cF_t] \leq4 \sum_{i=1}^{\kappa_t} w_i
(\E X_i^2 )\P(I_i) = 2 \sum_{i=1}^{\kappa_t} w_i \E X_i^2,
\]
and it remains to bound $\E X_i^2$. Following the same argument that
led to (\ref{eqXi}) now gives
\[
\E X_i^2 \leq\biggl(1 - \biggl(1 - \frac{w_i}{2} \biggr)^{\lfloor
\kappa
_t/2 \rfloor} \biggr) \biggl( \frac{2}{\kappa_t}\biggr)^2 +
\biggl(1 - \frac{w_i}{2} \biggr)^{\lfloor\kappa_t/2 \rfloor} w_i^2.
\]
As before, we now deduce that either $w_i \leq2/\kappa_t$, in which
case clearly $\E X_i^2 \leq4/\kappa_t^2$, or we have
\[
\E X_i^2 \leq( 1 - e^{-w_i\kappa_t/6})\frac{4}{\kappa_t^2}
+ e^{-w_i\kappa_t/6}w_i^2
\leq\frac1{\kappa_t^2} \bigl(4 + e^{-w_i\kappa_t/6} (w_i\kappa_t)^2
\bigr).
\]
Since $x^2\exp(-x/6) < 20$ for all $x\geq0$, it then follows that
$\E X_i^2 < 24/\kappa_t^2$ (with room to spare). Either way we deduce that
\[
\E[(\chi_{t+1}-\chi_t)^2 \mid\cF_t] < 2 \sum_i ( w_i \cdot
24/\kappa_t^2)= 48/\kappa_t^2
\]
and so, going back to (\ref{eq-Xi-sq}),
%
%
\begin{equation}
\label{eqXi2-60}
\E[(\Xi_{t+1})^2 \mid\cF_t] < 48 + 11 < 60.
\end{equation}
Using this bound in (\ref{eqY-CS}) we can conclude the proof as we have
\[
\E[ (Y_{t+1}-Y_t)^2 \mid\cF_t] < 3 \E[(\Xi_{t+1})^2 \mid\cF_t] +
\tfrac32M^2 < 2M^2.
\]
\upqed
\end{pf}

By now we have established that $(Y_t)$ is a supermartingale which
satisfies $Y_{t+1} - Y_t \leq L$ for a value of $L = \log^{2/3} n +
\frac{1}{200}$ and that, in addition, $\E[(Y_{t+1} -Y_t)^2 \mid\cF_t]
\leq2M^2$.
We are now in a position to apply the following inequality due to
Freedman~\cite{Freedman}; we note that this result
was originally stated for martingales yet its proof, essentially
unmodified, extends also to supermartingales.
%
%
\begin{theorem}[(\cite{Freedman}, Theorem 1.6)]
\label{thmfreedman}
Let $(S_i)$ be a supermartingale with respect to a filter
$(\cF_i)$. Suppose $S_{i}-S_{i-1} \leq L$ for all $i$, and write
$V_t = \sum_{i=1}^t \E[(S_i-S_{i-1})^2 \mid\cF_{i-1}]$. Then for
any $s,v>0$,
\[
\P(
\{S_t \geq S_0 + s, V_t \leq v\} \mbox{ for some $t$}
)
\leq
\exp\bigl( -\tfrac12 s^2/(v + L s ) \bigr).
\]
\end{theorem}

By the above theorem and a standard application of optional stopping,
for any $s>0$, integer $t$ and stopping time $\tau$ we have $\P( Y_{t
\wedge\tau} \geq Y_0 + s )
\leq\exp( -\frac12 s^2/(2M^2 t + Ls) ) $.
In particular, letting
\[
t_0 = 500 M \log n
\]
and plugging $s = \log^{3/4} n$ and $\tau= \bar{\tau}$ in the last
inequality we deduce that
\[
\P( Y_{t_0 \wedge\bar{\tau}} \geq Y_0 + \log^{3/4} n )
\leq
\exp\bigl( -\bigl(\tfrac12-o(1)\bigr)\log^{1/12} n \bigr) =o(1).
\]
Hence, recalling the value of $Y_0$ from (\ref{eq-Y0-def}) we have w.h.p.
%
%
\begin{equation}
\label{eq-Y-t0-bar-tau}
Y_{t_0 \wedge\bar{\tau}} \leq1 + M\log n + \log^{3/4}n \leq2M\log n
,
\end{equation}
where the last inequality holds for sufficiently large $n$.

In order to compare $t_0$ and $\bar{\tau}$, recall from
(\ref{eqXi-5}) that $\E[ \Xi_{t+1} \mid\cF_t ] \leq5$,
whereas we
established in (\ref{eqXi2-60}) that $\E[ (\Xi_{t+1})^2
\mid
\cF_t]
< 60$. By Chebyshev's inequality,
\[
\P( \Xi_{t+1} \geq\log^{2/3}n \mid\cF_t )
=
O(
\E[(\Xi_{t+1})^2 \mid\cF_t] \log^{-4/3} n
)
=
O( \log^{-4/3} n ).
\]
In particular, a union bound implies that
\[
\P(\bar{\tau} \leq t_0) = O(\log^{-1/3} n).
\]
Revisiting (\ref{eq-Y-t0-bar-tau}) this immediately implies that w.h.p.
\[
Y_{t_0}\leq2M \log n,
\]
and since $Y_{t_0 \wedge\bar{\tau} \wedge\tauc}
\geq Z_{t_0 \wedge\bar{\tau} \wedge\tauc}$ [due to (\ref{eqYZ})],
we further have that w.h.p.
\[
Y_{t_0 \wedge\tauc}
\geq
Z_{t_0 \wedge\tauc}
\geq
\frac{ t_0 \wedge\tauc}{200}.
\]
Therefore, we must have $\tauc< t_0$ w.h.p., otherwise the last two
inequalities would contradict our choice of $t_0 = 500 M \log n$. The
proof is complete.

\section{Super-logarithmic lower bound for the uniform process}
\label{secuniform}
In this section we use the analytic approximation framework introduced
by Schramm to prove the super-logarithmic lower bound stated in
Theorem~\ref{thm-1} for the coalescence time of the uniform process.
Recall that a key element in this framework is the normalized Laplace
transform of the cluster-size distribution, namely, $G_t(s) = (1/\kappa
_t) F_t(\kappa_t s)$, where $F_t(s) = \sum_{i=1}^{\kappa_t} e^{-w_i s}$
(see Definition~\ref{def-F-G}). The following proposition, whose proof
entails most of the technical difficulties in our analysis of the
uniform process, demonstrates the effect of $G_t(\frac12)$ and $G_t(1)$
on the coalescence rate.
%
%
\begin{proposition}
\label{propG-evol-conc}
Let $\epsilon_t = 1 - G_t( \frac{1}{2} )$ and $\zeta_t = G_t(1)$.
There exists an absolute constant $C>0$ such that, conditioned on $\cF
_t$, with probability at least $1- C \kappa_t^{-100}$, we have
%
%
\begin{eqnarray}
\label{eqK-evol-conc}
&\displaystyle |\kappa_{t+1} - ( 1 - \epsilon_t/2) \kappa_t| \leq
\kappa
_t^{2/3}, &\\
\label{eqG-evol-conc}
&\displaystyle  \zeta_{t+1} \geq\zeta_t + \epsilon_t^{13/\epsilon_t} - 8\kappa
_t^{-1/3}.&
\end{eqnarray}
\end{proposition}

We postpone the proof of this proposition to Section \ref
{secG-evol-conc} in favor of showing how the relations that it
establishes between $\kappa_t,G_t(1),G_t(\frac12)$ can be used to
derive the desired lower bound on $\tauc$.
We claim that as long as $\kappa_t,G_t(\frac12),G_t(1)$ satisfy
equations (\ref{eqK-evol-conc}), (\ref{eqG-evol-conc}) and $t = O
(\log n \cdot\frac{\log\log\log n}{\log\log n})$, then
$\kappa_t
\geq n^{3/4}$; this deterministic statement is given by the following lemma.

%
\begin{lemma}
\label{lemseq-evol}
Set $T = \frac{1}{75} \log n \cdot\frac{\log\log n}{\log\log\log
n}$ for a sufficiently large $n$
and let $\kappa_0,\ldots,\kappa_T$ be a sequence of integers in $\{
1,\ldots,n\}$ with $\kappa_0=n$.
Further, let $\epsilon_t$ and $\zeta_t$ for $t=0,\ldots,T$ be two sequences
of reals in $[0,1]$ and suppose that for all $t< T$ the three
sequences satisfy inequalities (\ref{eqK-evol-conc}) and (\ref
{eqG-evol-conc}). Then $\kappa_t > n^{3/4}$ for all $t\leq T$.
\end{lemma}

Observe that the desired lower bound on the coalescence time of the
uniform process $\mathcal{U}$ is an immediate corollary of Proposition
\ref{propG-evol-conc} and Lemma~\ref{lemseq-evol}. Indeed, condition on
the first $t$ rounds where $0 \leq t < T = \frac 1{75}\log
n\cdot\frac{\log\log n}{\log\log\log n}$ and assume $\kappa _t
>n^{3/4}$. Proposition~\ref{propG-evol-conc} implies that equations
(\ref{eqK-evol-conc}), (\ref{eqG-evol-conc}) hold\vspace*{1pt} except
with probability $O(\kappa_t^{-100})=o(n^{-1})$. In this event Lemma
\ref{lemseq-evol} yields $\kappa_{t+1}>n^{3/4}$, extending our
assumption to the next round. Accumulating these probabilities for all
$t<T$ now shows that $\P ( \kappa_T > n^{3/4}) = 1 - o(T/n)$ and in
particular $\tauc> T$ w.h.p., as required.
\begin{pf*}{Proof of Lemma~\ref{lemseq-evol}}
The proof proceeds by induction. Assuming that $\kappa_i > n^{3/4}$ for
all $i \leq t < T$, we wish to deduce that $\kappa_{t+1}>n^{3/4}$.

Repeatedly applying equation (\ref{eqG-evol-conc}) and using the
induction hypothesis we find that
%
%
\begin{eqnarray}
\label{eq-zeta-lower-bound}
\zeta_{t+1}
&\geq&
\zeta_0 + \sum_{i=0}^t (
\epsilon_i^{13/\epsilon_i} - 8 \kappa_i^{-1/3}
)
>
\sum_{i=0}^t (\epsilon_i^{13/\epsilon_i})
- 8(t+1) ( n^{3/4} )^{-1/3} \nonumber\\[-8pt]\\[-8pt]
&=& \sum_{i=0}^t(\epsilon
_i^{13/\epsilon_i}) - n^{-1/4+o(1)}\nonumber
\end{eqnarray}
since $t \leq T = n^{o(1)}$. Following this, we claim that the set $I =
\{ 0 \leq i \leq t\dvtx\epsilon_i \geq15\frac{\log\log\log
n}{\log
\log n}\}$ has size at most $(\log n)^{{9}/{10}}$. Indeed, as
$x^{1/x}$ is monotone increasing for all $x \leq e$, every such $i\in
I$ has
\begin{eqnarray*}
\epsilon_i^{13/\epsilon_i}
&\geq&
\biggl(
15\frac{\log\log\log n}{\log\log n}
\biggr)^{
{13\log\log n}/(15\log\log\log n)
}
\\
&=& (\log n)^{-{13}/{15}+o(1)} > (\log n)^{-{9}/{10}},
\end{eqnarray*}
where the last inequality holds for large $n$. Hence, if we had $|I| >
2(\log n)^{9/10}$
then it would follow from (\ref{eq-zeta-lower-bound}) that $\zeta_{t+1}
> 2-o(1)$, contradicting the assumption of the lemma for large enough $n$.

Moreover,\vspace*{1pt} by the assumption that $\epsilon_i \in[0,1]$, we
have $\frac 12 \leq(1-\epsilon_i/2)\leq1$ for all~$i$.
Together\vspace*{1pt} with the facts that $\kappa_{i+1}
\geq(1-\epsilon_i/2)\kappa_i -\kappa_i^{2/3}$ for all $i \leq t$ due to
(\ref{eqK-evol-conc}) while $\kappa_i \leq n$ for all $i$ we now get
\begin{eqnarray*}
\kappa_{t+1} &\geq&
\kappa_0 \prod_{i=0}^t ( 1 - \epsilon_i/2 )
- \sum_{i=0}^t \kappa_i^{2/3} \\
&\geq&
\biggl( 1 - \frac12 \cdot\frac{ 15\log\log\log n} {\log\log n}
\biggr)^t 2^{-|I|} n - (t+1) n^{2/3} \\
&\geq& e^{- 15({\log\log\log n})T/{\log\log n}} 2^{-|I|} n
- T n^{2/3},
\end{eqnarray*}
where the last inequality used the fact that $t < T$ as well as the
inequality $1 - x/2 > e^{-x}$, valid for all $0 < x < 1$.
Now, $2^{-|I|} = n^{-o(1)}$ since $|I| \leq2(\log n)^{9/10}$ and by
the definition of $T$ we obtain that
\[
\kappa_{t+1} \geq
e^{-(\log n)/5} n^{1 - o(1)} - n^{{2}/{3} + o(1)}
=
n^{{4}/{5} - o(1)}
>
n^{3/4}
\]
for sufficiently large $n$, as claimed. The proof is complete.
\end{pf*}

The remaining sections are devoted to the proof of Proposition \ref
{propG-evol-conc} and are organized as follows.
In Section~\ref{secGt12} we will relate $G_t(\frac12)$ to the
expected change in $\kappa_t$.
While unfortunately there is no direct recursive relation for the
sequence $\{ G_t(\frac{1}{2})\dvtx t\geq0\}$,
in Section~\ref{secrec-approx-F} we will approximate $\E[
F_{t+1}(\kappa_t s)\mid\cF_t]$
[closely related to $G_{t+1}(s) = (1/\kappa_{t+1}) F_{t+1}(\kappa_{t+1}
s)$] in terms of several evaluations of $G_t$.
We will then refine our approximation of $G_{t+1}(\frac12)$ in Section
\ref{secconvex-correc} by examining $F_{t+1}(s)$
at a point $s \approx\E[ \frac12\kappa_{t+1} \mid\cF_t]$.
Finally, these ingredients will be combined into the proof of
Proposition~\ref{propG-evol-conc} in Section~\ref{secG-evol-conc}.

\subsection{Relating $G_t(\frac12)$ to the coalescence rate}\label{secGt12}
The next lemma shows that the value of $G_t(\frac12)$ governs the
expected number of merges in round $t+1$.
%
%
\begin{lemma}
\label{lemEK}
Suppose that after $t$ rounds we have $\kappa_t \geq2$ clusters and
set $\epsilon_t = 1 - G_t(\frac12)$. Then
%
%
\begin{equation}
\label{eqEK}
\bigl|\E[ \kappa_{t+1} \mid\cF_t ] -
(1-\epsilon_t/2) \kappa_t \bigr| \leq\tfrac{1}{4}.
\end{equation}
\end{lemma}

This emphasizes the importance of tracking the value of
$G_t(\frac{1}{2})$, as one could derive a lower bound on the
coalescence time by showing that $G_t(\frac{1}{2})$ is sufficiently
close to 1 (i.e., $\epsilon_t$ is suitably small).
In order to prove this lemma we first require two straightforward facts
on the functions involved.
%
%
\begin{claim}
\label{clmG-lip}
The following holds for all $t$ with probability $1$. The function
$G_t(\cdot)$ is convex, decreasing and $1$-Lipschitz on the domain
$\mathbb{R}^+$. Furthermore, $G_t(s) \geq e^{-s}$ for any $s$.
\end{claim}
\begin{pf}
Denote the cluster-sizes at the end of round $t$ by $w_1,\ldots
,w_{\kappa_t}$.
Recall that by definition $G_t(s) = (1/\kappa_t)F_t(\kappa_t s) =
(1/\kappa_t) \sum_i e^{-w_i \kappa_t s}$ is an arithmetic mean of
negative exponentials of $s$, hence, convex and
decreasing. Moreover, its first derivative is $G'_t(s) = F'_t(\kappa_t s)$
and in particular
\[
G_t'(0) = F_t'(0) = -\sum_i w_i = -1.
\]
Since $G_t'(s)$ is increasing and negative we deduce that $G_t$ is
indeed 1-Lipschitz. Finally, since the negative exponential function is
convex, Jensen's
inequality concludes the proof by yielding
\[
G_t(s)
=
\frac{1}{\kappa_t}
\sum_i e^{-w_i \kappa_t s}
\geq
e^{ - (1/\kappa_t) \sum_i w_i \kappa_t s }
= e^{-s}.
\]
\upqed
\end{pf}
%
%
\begin{claim}
\label{clmexp-approx}
For any real numbers $0 \leq x \leq1$ and $\kappa> 0$ we have
$ (1-x)^\kappa\geq e^{-\kappa x} - (e \kappa)^{-1}$.
\end{claim}
\begin{pf} Fix $\kappa> 0$ and consider the function $f(x) =
\kappa( e^{-\kappa x} - (1-x)^\kappa)$.
The desired inequality is equivalent to having $f(x) \leq1/e$ for all
$0 \leq x \leq1$,
hence, it suffices to bound $f(x)$ at all local maxima, then compare that
bound to its values at the endpoints $f(0) = 0$ and $f(1) = \kappa
e^{-\kappa}$.

It is easy to verify that any local extrema $x^*$ must satisfy
$ (1-x^*)^{\kappa- 1} = e^{-\kappa x^*} $, and so
%
\[
f(x^*) = \kappa e^{-\kappa x^*} \bigl(1 - (1-x^*)\bigr) = (\kappa x^*) e^{-\kappa
x^*}.
\]
Since $y e^{-y} \leq1/e$ for any $y\in\R$, both $f(x^*)$ and $f(1)$
are at most $1/e$, as required.
\end{pf}
\begin{pf*}{Proof of Lemma~\ref{lemEK}}
Let $\kappa= \kappa_t$ and $\kappa' = \kappa_{t+1}$, and as usual
let $w_1, \ldots,
w_\kappa$ denote the cluster-sizes at the end of $t$ rounds.
Recalling the definition of the uniform coalescence process,
the number of pairs of clusters that merge in round $t+1$ is equal to
the number of clusters which:
\begin{longlist}[(a)]
\item[(a)] select to be acceptors in this round, and
\item[(b)] receive at least one incoming request in this round.\vadjust{\goodbreak}
\end{longlist}
(Compare this simple characterization with the number of merges in the
\textit{size-biased} process,
where one must also consider the cluster-sizes of the incoming requests relative
to the size of the acceptor.)
A given cluster $\cC_i$ becomes an acceptor with probability $\frac12$,
and\vspace*{1pt} conditioning on this event
we are left with $\kappa-1$ other clusters, each of which may send a
request to the cluster $\cC_i$ with
probability $w_i/2$ (the factor of $2$ accounts for the choice to issue
rather than accept requests this round)
independently of its peers.
Altogether we conclude that the probability that $\cC_i$ accepts an
incoming request is
exactly $\frac{1}{2} ( 1 - (1 - w_i/2)^{\kappa-1} )$ and so the
expected total number of merges is
\[
\E[ \kappa- \kappa' \mid\cF_t ]
=
\frac12 \sum_i \bigl(
1 - (
1 - w_i/2
)^{\kappa-1}
\bigr).
\]
Therefore,
\begin{eqnarray*}
\E[ \kappa- \kappa' \mid\cF_t ]
&\geq&
\frac{1}{2} \sum_i \bigl(
1 - e^{- w_i (\kappa- 1)/2 }
\bigr) 
=
\frac{1 - G_t (({\kappa- 1})/({2\kappa})
) }
{2} \kappa\\
&\geq&
\frac{1 - G_t ( {1}/{2} ) }{2} \kappa
-
\frac{1}{4},
\end{eqnarray*}
where the last inequality is due to $G_t$ being 1-Lipschitz as was
established in Claim~\ref{clmG-lip}.
For an upper bound on the expected number of merges we apply Claim~\ref
{clmexp-approx}, from which it follows that
\begin{eqnarray*}
\E[ \kappa- \kappa' \mid\cF_t ]
&\leq&
\frac{1}{2} \sum_i
\biggl(
1 - e^{-w_i (\kappa-1)/2 } + \frac{1}{e \kappa}
\biggr) 
\leq
\frac{1}{2} \sum_i
(
1 - e^{- w_i \kappa/2}
)
+ \frac{1}{2 e} \\
&=&
\frac{1 - G_t ( {1}/{2} ) }{2} \kappa
+
\frac{1}{2 e}.
\end{eqnarray*}
Combining these bounds gives the required result.
\end{pf*}

\subsection{Recursive approximation for $F_t$}\label{secrec-approx-F}
Despite the fact that there is no direct recursion for the values of
$G_t(\frac12)$,
it turns out that on the level of expectation one can
recover values of its counterpart $F_{t+1}$ from several different
evaluations of
$G_t$. Note that this still does not provide an estimate for the
expected value of $G_{t+1}$, as
the transformation between the $F_{t+1}$ and $G_{t+1}$ unfortunately
involves the number of
clusters at time $t+1$, thereby introducing nonlinearity to the approximation.
%
%
\begin{lemma}
\label{lemFG-evol}
Suppose that after $t$ rounds $\kappa_t \geq2$ and let $\epsilon_t =
1-G_t(\frac12)$. Then
%
%
\begin{equation}
\label{eqFG-evol}
\E[F_{t+1}(\kappa_t s) \mid\cF_t]
>
(1-\epsilon_t/2)
\kappa_t
\biggl[
\frac{\alpha}{\alpha+\beta} G_t(s)
+
\frac{\beta}{\alpha+ \beta} G_t \biggl( s+ \frac{1}{2} \biggr)
\biggr]
- 2,\hspace*{-35pt}
\end{equation}
where
\[
\alpha= \alpha(s,t) = G_t(s) + G_t \bigl( \tfrac{1}{2} \bigr),\qquad
\beta=\beta(s,t)= 1 - G_t(s).\vadjust{\goodbreak}
\]
\end{lemma}
\begin{remark*}
Although the approximation in (\ref{eqFG-evol}) may look
intractable, its structure is in fact quite useful. The leading factor
$(1-\epsilon_t/2) \kappa_t$ is essentially
$\E[\kappa_{t+1} \mid\cF_t]$ from Lemma~\ref{lemEK}, which is
particularly convenient as we will need to divide by $\kappa_{t+1}$ to
pass from
$F_{t+1}$ to $G_{t+1}$.
\end{remark*}
\begin{pf*}{Proof of Lemma~\ref{lemFG-evol}}
As stated before, let $\kappa= \kappa_t$ and denote the cluster-sizes
by $w_1, \ldots, w_\kappa$. We
account for the change $F_{t+1}(s) - F_t(s)$ as follows. Should the
clusters $\cC_i$ and $\cC_j$
merge in round $t+1$, this would contribute exactly $e^{-(w_i + w_j)s}
- e^{-w_i s} -
e^{-w_j s}$ to $F_{t+1}(s)-F_t(s)$. Thus, $\E[F_{t+1}(s)-F_t(s)\mid
\cF_t]$
is simply the sum of these expressions, weighted by the probabilities
that the individual pairs
merge.

Let us calculate the probability that $\cC_i$ accepts an
incoming request from the cluster $\cC_j$. First let $R_i$ denote the
event that $\cC_i$ accepts an incoming request from \textit{some} cluster,
which was shown in the proof of Lemma~\ref{lemEK} to satisfy
$\P(R_i \mid\cF_t) = \frac{1}{2} ( 1 - (1 - w_i/2)^{\kappa-1}
)$.
Crucially, the fact that acceptors select an incoming request to merge with
via a uniform law now implies that, given $R_i$, the identity of the
cluster that $\cC_i$ merges with is uniform over the remaining $\kappa
-1$ clusters by symmetry. In particular, the probability that $\cC_i$
accepts a merge request from $\cC_j$ equals $\P(R_i \mid\cF
_t)/(\kappa
-1)$ and so
%
%
\begin{eqnarray}
\label{eqF-evol-1}\quad
&&
\E[F_{t+1}(s) - F_t(s) \mid\cF_t]\nonumber\\[-8pt]\\[-8pt]
&&\qquad=
\sum_{i \neq j}
\bigl( e^{-(w_i + w_j)s} - e^{-w_i s} - e^{-w_j s} \bigr)
\frac{1}{2} \bigl(
1 - (
1 - w_i/2
)^{\kappa-1}
\bigr)
\frac{1}{\kappa-1}.\nonumber
\end{eqnarray}
The term $( 1 - w_i/2 )^{\kappa-1}$ is greater or equal to
$e^{-(\kappa-1) w_i/2} - [e(\kappa-1)]^{-1} \geq e^{- w_i \kappa/2} -
[e(\kappa-1)]^{-1} $
by Claim~\ref{clmexp-approx}. Since $e^{-(w_i + w_j)s} - e^{-w_i s} -
e^{-w_j s}$ is always
negative by convexity, this gives
\begin{eqnarray*}
&&
\E[F_{t+1}(s) - F_t(s) \mid\cF_t]\\
&&\qquad\geq
\sum_{i \neq j}
\bigl( e^{-(w_i + w_j)s} - e^{-w_i s} - e^{-w_j s} \bigr)
\frac{1}{2} \biggl(
1 -
e^{-w_i\kappa/2} + \frac{1}{e(\kappa-1)}
\biggr)
\frac{1}{\kappa-1}.
\end{eqnarray*}
Next, observe that
%
%
\begin{equation}
\label{eqexp-lip}\quad
e^{-(w_i + w_j)s} - e^{-w_i s} - e^{-w_j s}
= -1 +
( 1 - e^{-w_i s} )
( 1 - e^{-w_j s} )
\geq-1,
\end{equation}
hence, we can sum the effect of the term $1/(e(\kappa-1))$ over all
$\kappa(\kappa-1)$ indices $i\neq j$ and get
\begin{eqnarray*}
&&\E[F_{t+1}(s) - F_t(s) \mid\cF_t]\\
&&\qquad\geq
\frac{1}{2(\kappa-1)} \sum_{i \neq j}\bigl[
\bigl( e^{-(w_i + w_j)s} - e^{-w_i s} - e^{-w_j s} \bigr)
(1 - e^{-w_i \kappa/2}) \bigr]\\
&&\qquad\quad{} - \frac{\kappa}{2e(\kappa-1)}.
\end{eqnarray*}
Note that the last expression has magnitude at most $1/e$ due to the
assumption $\kappa\geq2$.
Furthermore, each of the $\kappa(\kappa-1)$ summands in the summation
over $i\neq j$ has magnitude at most $1$,
hence, we may replace the factor $1/(\kappa-1)$
with $1/\kappa$ in front of the summation at a maximal cost of
$\frac12( \frac1{\kappa-1} - \frac1{\kappa}) \kappa(\kappa
-1)=\frac12$, giving
%
%
\begin{eqnarray}\label{eqFG-evol-2}
&&\E[F_{t+1}(s) - F_t(s) \mid\cF_t]\nonumber\\[-3pt]
&&\qquad\geq
\frac{1}{2\kappa} \sum_{i \neq j}\bigl[
\bigl( e^{-(w_i + w_j)s} - e^{-w_i s} - e^{-w_j s} \bigr)
(
1 - e^{-w_i\kappa/2}
) \bigr]
- \frac12 - \frac{1}{e} \\[-3pt]
&&\qquad> \frac{1}{2\kappa} \sum_{i,j}\bigl[
\bigl( e^{-(w_i + w_j)s} - e^{-w_i s} - e^{-w_j s} \bigr)
(
1 - e^{-w_i\kappa/2}
) \bigr]
- 2,\nonumber
\end{eqnarray}
where the last inequality is due to each of the $\kappa$ diagonal
terms $i=j$ having magnitude at most 1.

Since (\ref{eqFG-evol}) addresses $F_{t+1}(\kappa s)$ rather than $F_{t+1}(s)$
we now focus on the following summation:
\begin{eqnarray*}
&&\sum_{i,j} \bigl( e^{-(w_i + w_j) \kappa s} - e^{-w_i \kappa s} -
e^{-w_j \kappa s} \bigr) ( 1 - e^{-w_i\kappa/2} ) \\[-3pt]
&&\qquad= \sum_{i,j} \bigl( e^{-w_i \kappa s - w_j \kappa s} - e^{-w_i
(\kappa
s + {\kappa}/{2}) - w_j \kappa s}
- e^{-w_i \kappa s}\\[-3pt]
&&\qquad\quad\hspace*{26.5pt}{} +
e^{-w_i (\kappa s + {\kappa}/{2})}
- e^{-w_j \kappa s}
+
e^{- w_j \kappa s -w_i \kappa/2 } \bigr)\\[-3pt]
&&\qquad=
F_t(\kappa s)^2
- F_t\biggl(\kappa s + \frac{\kappa}{2} \biggr) F_t(\kappa s) - \kappa
F_t(\kappa s)
+ \kappa F_t\biggl(\kappa s + \frac{\kappa}{2} \biggr)\\[-3pt]
&&\qquad\quad{} - \kappa
F_t(\kappa s)
+ F_t\biggl(\frac{\kappa}{2} \biggr) F_t(\kappa s).
\end{eqnarray*}
Using this in (\ref{eqFG-evol-2}), noting that the term $- F_t(\kappa
s)$ cancels out, we find that
\begin{eqnarray*}
&&
\E[ F_{t+1}(\kappa s) \mid\cF_t ]\\[-3pt]
&&\qquad>
\frac{1}{2\kappa} \biggl[
F_t(\kappa s)^2
- F_t\biggl(\kappa s + \frac{\kappa}{2} \biggr) F_t(\kappa s)\\[-3pt]
&&\qquad\quad\hspace*{18.5pt}{}
+ \kappa F_t\biggl(\kappa s + \frac{\kappa}{2} \biggr)
+ F_t\biggl(\frac{\kappa}{2} \biggr) F_t(\kappa s)
\biggr] - 2 \\[-3pt]
&&\qquad=
\frac{ \kappa+ F_t( {\kappa}/{2} ) }{2}
\biggl[
\frac{ F_t(\kappa s) + F_t ( {\kappa}/{2} ) }
{ \kappa+ F_t ( {\kappa}/{2} ) }
\cdot\frac{F_t(\kappa s)}{\kappa}\\[-3pt]
&&\qquad\quad\hspace*{63.5pt}{} +
\frac{ \kappa- F_t(\kappa s) }
{ \kappa+ F_t ( {\kappa}/{2} ) }
\cdot\frac{F_t ( \kappa s + {\kappa}/{2} )
}{\kappa}
\biggr] - 2\\[-3pt]
&&\qquad=
\frac{ 1 + G_t( {1}/{2} ) }{2} \kappa
\biggl[
\frac{\alpha}{\alpha+\beta} G_t(s)
+
\frac{\beta}{\alpha+ \beta} G_t \biggl( s+ \frac{1}{2} \biggr)
\biggr]
- 2,
\end{eqnarray*}
where $\alpha= G_t(s) + G_t ( \frac{1}{2} )$ and
$\beta= 1 - G_t(s)$, thus establishing (\ref{eqFG-evol}).\vadjust{\goodbreak}
\end{pf*}

\subsection{Quantifying the convexity correction in the recursion for $F_t$}\label{secconvex-correc}
Examine the recursion established in Lemma~\ref{lemFG-evol}.
In order to derive lower bounds on the~$F_t$'s,
we recognize the second factor in the right-hand side of (\ref
{eqFG-evol}) as a weighted arithmetic mean of
two evaluations of $G_t$. Recalling that $G_t$ is a convex combination of
negative exponentials, we will now estimate the ``convexity correction''
between $G_t$ and its weighted mean. It is precisely this increment
which will allow us to show that $G_t$ rises toward 1
at a nontrivial rate, as the following lemma demonstrates.
%
%
\begin{lemma}
\label{lemE-increment}
Suppose after $t$ rounds $\kappa_t \geq2$ and let $\epsilon_t = 1 -
G_t(\frac{1}{2})$ and $\kappa^* = (1-\epsilon_t/2) \kappa_t$.
Then
%
%
\begin{equation}
\label{eqE-increment}
\E[ F_{t+1} (\kappa^*) \mid\cF_t ]
\geq
[
G_t(1) + \epsilon_t^{13/\epsilon_t}
] \kappa^*
-2.
\end{equation}
\end{lemma}

Indeed, by Lemma~\ref{lemEK} we recognize that $\kappa^*$ is
approximately $\E[\kappa_{t+1}\mid\cF_t]$,
hence, postponing for the moment concentration arguments, one sees that
equation (\ref{eqE-increment})
resembles the form of (\ref{eqG-evol-conc}). 
Our first step in proving this lemma will be to
establish a lower bound similar to (\ref{eqE-increment}) which
replaces the $\epsilon_t^{1/\epsilon_t}$ term
by the convexity correction between $G_t$ and its weighted mean from
(\ref{eqFG-evol}).
%
%
\begin{claim}
\label{clmG1+cvx}
Suppose after $t$ rounds $\kappa_t \geq2$ and let $\epsilon_t = 1 -
G_t(\frac{1}{2})$ and $\kappa^* = (1-\epsilon_t/2) \kappa_t$.
Let $h(s)$ be the secant line intersecting $G_t(s)$ at $s_1 =
\kappa^*/\kappa_t$ and $s_2 = s_1 + \frac{1}{2}$. Let
$ \theta= \frac{\alpha}{\alpha+\beta} s_1 + \frac{\beta}{\alpha+
\beta} s_2 $
where $\alpha=G_t(s_1)+G_t(\frac12)$ and $\beta=1-G_t(s_1)$, and let
$\Delta= h(\theta) - G_t(\theta)$. Then
%
%
\begin{equation}
\label{eqG1+cvx-2}
\E[ F_{t+1}(\kappa^*) \mid\cF_t ]
\geq
[ G_t(1) + \Delta] \kappa^*
-2
\end{equation}
and in addition
%
%
\begin{equation}
\label{eqG1+cvx-3}
\frac{\epsilon_t}{4}
\leq
\theta- s_1
\leq
\frac{1}{4}.
\end{equation}
\end{claim}
\begin{pf}
Applying Lemma~\ref{lemFG-evol} with $s = s_1 $ and rewriting its
statement in terms of $h,\theta,\Delta$ give
\[
\E[F_{t+1}(\kappa^*) \mid\cF_t]
>
(1-\epsilon_t/2) \kappa_t
h(\theta)
- 2
=
h(\theta) \kappa^* - 2
=
[ G_t(\theta) + \Delta]\kappa^* - 2.
\]
Since we established in Claim~\ref{clmG-lip} that $G_t$ is decreasing,
(\ref{eqG1+cvx-2}) will follow from showing that $\theta\leq1$.
Note that $\theta$ is a weighted mean between $s_1 = 1-\epsilon_t/2$
and $s_2 = s_1 + \frac12$,
and so it is not immediate that $\theta\leq1$. To show that this is
the case, we argue as follows.

Recalling the definition of $\theta$, we wish to show that $\alpha s_1
+ \beta(s_1+\frac12) \leq\alpha+\beta$
where $\alpha=G_t(s_1)+G_t(\frac12)$ and $\beta=1-G_t(s_1)$. Observe
that $\alpha+ \beta= 1+G_t(\frac12) = 2-\epsilon_t = 2s_1$ by
definition.\vadjust{\goodbreak} Therefore, $\theta\leq1$ if and only if $(\alpha+\beta
)s_1 + \beta/2 \leq2s_1$, or equivalently
\[
2\bigl(G_t\bigl(\tfrac12\bigr)-1\bigr)s_1 + 1 \leq G_t(s_1).
\]
We claim that indeed
%
%
\begin{equation}
\label{eq-theta-ineq}
2\bigl(G_t\bigl(\tfrac12\bigr)-1\bigr)s + 1 \leq G_t(s) \qquad\mbox{for any $s
\geq\frac12$},
\end{equation}
which would, in particular, imply that it holds for $s=s_1 > \frac12$
since $s_1 = 1-\epsilon_t/2$ with $\epsilon_t < 1$. In order to
verify (\ref{eq-theta-ineq}) observe that its left-hand side is an
affine function of $s$ whereas the right-hand side is convex and that
equality holds for $s=0$ [recall that $G_t(0)=1$] and $s=\frac12$.
Thus, the affine left-hand side does not exceed the convex right-hand
side for any $s \geq\frac{1}{2}$, as required. We now conclude that
$\theta\leq1$, establishing (\ref{eqG1+cvx-2}).

It remains to prove (\ref{eqG1+cvx-3}). Since $\theta$ is a
weighted arithmetic mean of $s_1$ and $s_2=s_1 + \frac{1}{2}$, the
upper bound will follow
once we show that the weight on $s_1$ exceeds the
weight on $s_2+\frac12$, that is, when $\alpha> \beta$ or equivalently
\[
2 G_t (s_1) + G_t \bigl( \tfrac{1}{2} \bigr) > 1.
\]
This indeed holds, as Claim~\ref{clmG-lip} established that $G_t(s)
\geq e^{-s}$ and therefore the left-hand side above is at least
$2e^{-s_1}+ e^{-1/2} \geq2/e + e^{-1/2} >\frac54$, where we used the
fact that $s_1 = 1-\epsilon_t/2 \leq1$.

For the lower bound in (\ref{eqG1+cvx-3}), recall from Claim \ref
{clmG-lip} that $G_t$ is decreasing and $G_t(0)=1$,
which together with the aforementioned fact that $s_1 \geq\frac12$ gives
\[
\frac{\beta}{\alpha+\beta} = \frac{1-G_t(s_1)}{1+G_t(1/2)}
\geq
\frac{1-G_t(1/2)}2 = \frac{\epsilon_t}2.
\]
The proof is now concluded by noting that $\theta- s_1 = \frac12
\cdot
\frac{\beta}{\alpha+\beta}$ by definition.~%
\end{pf}

Next, we will provide a lower bound on the convexity correction in
terms of the difference between two
evaluations of $G_t$.
%
%
\begin{claim}
\label{clmcvx-correction}
Let\vspace*{1pt} $s \leq1$ and let
$h$ be the secant line intersecting $G_t$ at $s$ and $s +
\frac{1}{2}$. For any $0 \leq\delta\leq\frac{1}{4}$,
\[
h(s + \delta) - G_t(s + \delta)
\geq
\frac{\delta^2}{2}
\biggl[
G_t\biggl( \frac{1}{2}\biggr) - G_t(1)
\biggr]^2.
\]
\end{claim}
\begin{pf}
Let $g$ denote the secant line intersecting $G_t$ at $s$ and $s +
2\delta$. Since $\delta\leq\frac{1}{4}$ and $G_t$ is a decreasing
convex function,
\[
G_t(s + \delta) < g(s + \delta) \leq h(s + \delta).
\]
It thus suffices to show the following to deduce the statement of the claim:
%
%
\begin{equation}
\label{eqcvx-correction-2}
g(s + \delta) - G_t(s + \delta) \geq\frac{\delta^2}{2}
\biggl[
G_t\biggl( \frac{1}{2}\biggr) - G_t(1)
\biggr]^2,\vadjust{\goodbreak}
\end{equation}
which has a particularly convenient left-hand side due to the fact
that $g(s+\delta) = \frac12[G_t(s)+G_t(s+2\delta)]$ by definition.
Now let $\kappa= \kappa_t$ and let $w_1,\ldots,w_\kappa$ be the
cluster-sizes at the end of round $t$. We have
%
%
\begin{eqnarray}\label{eqg-G}
&&
\frac{1}{2}[G_t(s) + G_t(s + 2\delta)] - G_t(s + \delta)\nonumber\\
&&\qquad=
\frac{1}{2\kappa} \sum_i \bigl[
e^{-w_i \kappa s}
- 2e^{-w_i \kappa(s + \delta)}
+ e^{-w_i \kappa(s + 2\delta)}
\bigr] \\
&&\qquad=
\frac{1}{2\kappa} \sum_i e^{-w_i \kappa s} (
1
- e^{-w_i \kappa\delta}
)^2.\nonumber
\end{eqnarray}
By Cauchy--Schwarz, the right-hand side of (\ref{eqg-G}) satisfies
%
%
\begin{eqnarray}\label{eqcvx-correction-3}
\frac{1}{2\kappa} \sum_i e^{-w_i \kappa s} (
1
- e^{-w_i \kappa\delta}
)^2
&\geq&
\frac12
\biggl[
\frac{1}{\kappa} \sum_i e^{-w_i \kappa s/2} ( 1 - e^{-w_i
\kappa
\delta} )
\biggr]^2\nonumber\\[-8pt]\\[-8pt]
&=&
\frac{1}{2} \biggl[
G_t \biggl( \frac{s}{2} \biggr)
- G_t \biggl( \frac{s}{2} + \delta\biggr)
\biggr]^2.\nonumber
\end{eqnarray}
Set $K = \lceil1/2\delta\rceil$, noting that $K \leq1/\delta$ as
$\delta\leq\frac14$.
Since $G_t$ is a decreasing convex function we have
\[
G_t \biggl( \frac{s}{2} \biggr) - G_t \biggl( \frac{s}{2} + \delta
\biggr)
\geq G_t \biggl( \frac{s}{2} + (j-1)\delta\biggr) - G_t \biggl(
\frac
{s}{2} + j\delta\biggr) \qquad\mbox{for any $j \geq1$},
\]
and summing these equations for $j=1,\ldots,K$ yields
\begin{eqnarray*}
G_t \biggl( \frac{s}{2} \biggr) - G_t \biggl( \frac{s}{2} + \delta
\biggr)
&\geq&
\frac{1}{K} \biggl[
G_t \biggl( \frac{s}{2} \biggr)
- G_t \biggl( \frac{s}{2} + K \delta\biggr)
\biggr]\\
&\geq&
\frac{1}{K} \biggl[
G_t \biggl( \frac{s}{2} \biggr)
- G_t \biggl( \frac{s}{2} + \frac{1}{2} \biggr)
\biggr],
\end{eqnarray*}
which is at least $(1/K) [ G_t ( \frac{1}{2}) - G_t(1) ]$
once again since $s \leq1$ and $G_t$ is convex and decreasing.
Therefore, since $K \leq1/\delta$ we can conclude that
\[
G_t \biggl( \frac{s}{2} \biggr)
- G_t \biggl( \frac{s}{2} + \delta\biggr)
\geq
\delta\biggl[ G_t \biggl( \frac{1}{2}\biggr) - G_t (1) \biggr],
\]
which together with (\ref{eqg-G}), (\ref{eqcvx-correction-3}) now
establishes (\ref{eqcvx-correction-2})
and thus the proof is complete.
\end{pf}

The above claim quantified the convexity correction in terms of
$G_t(\frac{1}{2}) - G_t(1)$, and next we wish to estimate this quantity
in terms of the key parameter $\epsilon_t = 1-G_t(\frac12)$, which governs
the coalescence rate as was established by Lemma~\ref{lemEK}.
%
%
\begin{claim}
\label{clmG5-G1}
For any $t$ we have $G_t( \frac{1}{2} ) - G_t(1) \geq\epsilon
_t^{5/\epsilon_t}$,
where $\epsilon_t = 1 - G_t( \frac{1}{2} )$.
\end{claim}
\begin{pf}
We first claim that
%
%
\begin{equation}
\label{eqG-CS}
G_t(s) - G_t(2s)
\leq
\sqrt{G_t(2s) - G_t(4s)} \qquad\mbox{for any $s > 0$}.
\end{equation}
Indeed, let $\kappa=\kappa_t$, let $w_1,\ldots,w_\kappa$ be the
cluster-sizes after time $t$ and define
\begin{eqnarray*}
X &=& G_t(0) - G_t(s) = \frac{1}{\kappa} \sum_i ( 1 - e^{- w_i
\kappa s} ), \\
Y &=& G_t(s) - G_t(2s) = \frac{1}{\kappa} \sum_i e^{- \kappa s} (
1 - e^{- w_i \kappa s} ), \\
Z &=& G_t(2s) - G_t(3s) = \frac{1}{\kappa} \sum_i e^{-2\kappa s}
(
1 - e^{- w_i \kappa s} ).
\end{eqnarray*}
By Cauchy--Schwarz, $Y \leq\sqrt{XZ}$. Moreover, $X Z\leq Z \leq
G_t(2s) - G_t(4s)$
since $G_t$ is decreasing and $G_t(0) = 1$, and combining these
inequalities now establishes~(\ref{eqG-CS}).

Let $\gamma= G_t( \frac{1}{2} ) - G_t(1)$ and let $r \geq2$. A
repeated application of (\ref{eqG-CS}) reveals that
\[
G_t (2^{-k})
- G_t \bigl(2^{-(k-1)}\bigr)
\leq\gamma^{1/2^{k-1}}
\qquad\mbox{for $k = 1,2,\ldots,r$},
\]
%
and summing these equations we find that
\[
G_t ( 2^{-r} ) - G_t \biggl( \frac{1}{2} \biggr)
\leq\sum_{k=1}^r \gamma^{1/2^{k-1}} \leq r \gamma^{1/2^{r-1}}.
\]
On the other hand, since $G_t$ is
1-Lipschitz we also have $G_t ( 2^{-r}) \geq G_t(0) - 2^{-r} = 1 - 2^{-r}$.

At this point, recalling that $\epsilon_t = 1 - G_t( \frac{1}{2} )$ and
combining it with the above bounds gives
%
%
\begin{equation}
\label{eq-epsilon-gamma}
\epsilon_t - 2^{-r} \leq G_t(2^{-r}) - G_t\bigl(\tfrac
12\bigr) \leq r \gamma^{1/2^{r-1}}.
\end{equation}
The above inequality is valid for any integer $r \geq2$ and we now
choose $r = \lceil\log_2(4/3\epsilon_t)\rceil$, or
equivalently $r$ is the least integer such that $2^{-r} \leq\frac34
\epsilon_t$. One should notice that indeed $r\geq2$ since we have
$\epsilon_t < \frac12$, which in turn follows from the fact $G_t(s)
\geq e^{-s}$ (see Claim~\ref{clmG-lip}) yielding
%
%
\begin{equation}
\label{eqG5}
\epsilon_t \leq1 - e^{-1/2} < \tfrac25.
\end{equation}
Revisiting (\ref{eq-epsilon-gamma}) and using the fact that $2^{-r}
\leq\frac34\epsilon_t$, we find that $ \epsilon_t/4 \leq r
\gamma^{1/2^{r-1}}$ and after rearranging $\gamma\geq(\epsilon_t /
4r)^{2^{r-1}}$. Moreover,\vspace*{2pt} by definition $r
\leq\log_2(8/3\epsilon_t)$ and as one can easily verify that $
4\log_2(8/3x) < x^{-11/4}$ for all $0<x \leq \frac 25$ [which by
(\ref{eqG5}) covers the range of $\epsilon_t$], we have $r < \frac14
\epsilon_t^{-11/4}$. The choice of $r$ further implies that $2^{r-1} <
4/3\epsilon_t$ and combining these bounds gives
\[
\gamma> \biggl( \frac{\epsilon_t}{4 r} \biggr)^{4/3\epsilon_t} >
(\epsilon_t^{15/4})^{4/3\epsilon_t} = \epsilon_t^{5/\epsilon
_t}
\]
as claimed.
\end{pf}

We are now ready to establish equation (\ref{eqE-increment}),
the quantitative bound on the convexity correction in the weighted mean
of (\ref{eqFG-evol}).
\begin{pf*}{Proof of Lemma~\ref{lemE-increment}}
By Claim~\ref{clmG1+cvx}, in order to prove (\ref
{eqE-increment}) it suffices to show
that $\Delta\geq\epsilon_t^{13/\epsilon_t}$ with $\Delta$ as defined
in the statement of that claim. Using (\ref{eqG1+cvx-3}) of
Claim~\ref{clmG1+cvx} we
can write $\Delta= h(s_1 + \delta) - G_t(s_1 + \delta)$ where $h$ is
the secant line defined in that claim,
$s_1 = 1-\epsilon_t/2$ and $\delta$ satisfies $\epsilon_t\leq
4\delta
\leq1$.
Therefore, Claim~\ref{clmcvx-correction} implies that
$ \Delta\geq\frac12 (\epsilon_t/4)^2 [ G_t (\frac{1}{2}) - G_t(1)]^2$.
Applying Claim~\ref{clmG5-G1} we find that
\[
\Delta
\geq
\frac{1}{2}
\biggl(
\frac{\epsilon_t}{4}
\biggr)^2
(
\epsilon_t^{5/\epsilon_t}
)^2
\geq
\epsilon_t^{13/\epsilon_t},
\]
where we consolidated the constant factors into the
exponent using the fact that $x^2/32 > x^{3/x}$ for all $0< x \leq
\frac
25$ while bearing in mind that by (\ref{eqG5}) indeed $\epsilon_t <
\frac25$.
\end{pf*}

\subsection{\texorpdfstring{Proof of Proposition \protect\ref{propG-evol-conc}}{Proof of Proposition 5.1}}
\label{secG-evol-conc}
Let $\kappa= \kappa_t$ and note that w.l.o.g. we may assume that
$\kappa$ is sufficiently large
by choosing the constant $C$ from the statement of the proposition
appropriately.

Let $w_1, \ldots, w_\kappa$ denote the cluster-sizes.
As argued before, given $\cF_t$ one can realize round $t+1$ of the
process by a $\kappa$-dimensional
product space, where clusters behave independently as follows:
\begin{longlist}[(3)]
\item[(1)] For each $i$, the cluster $\cC_i$ decides whether to send or
accept requests via a fair coin toss.
\item[(2)] When sending a request $\cC_i$ selects its recipient cluster
randomly (proportionally to the $w_j$'s).
\item[(3)] When accepting requests $\cC_i$ generates a random real number
between 0 and~1 to be used to select the incoming merge-request it will
grant (uniformly over all the incoming requests).
\end{longlist}
As such, conditioned on $\cF_t$ the variable $\kappa_{t+1}$ is clearly
1-Lipschitz w.r.t. the above product space since changing the value
corresponding to the action of one cluster can affect at most one merge.
Thus, by a standard well-known coupling argument (see, e.g.,~\cite{AS})
the increments of the corresponding Doob martingale are bounded by 1
(i.e., $|M_{i+1}-M_i|\leq1$ where $M_i = \E[ \kappa_{t+1} \mid\cF
'_i]$ with $\cF'_i$ being the $\sigma$-algebra generated by the actions
of clusters $1,\ldots,i$ and $\cF_t$). Hoeffding's inequality now gives
\[
\P\bigl( \bigl|\kappa_{t+1} - \E[\kappa_{t+1} \mid\cF_t]\bigr| >
a
\mid \cF_t\bigr) \leq2\exp(-a^2 / 2\kappa) \qquad\mbox{for any
$a>0$}.
\]
Letting $\kappa^* = (1+\epsilon_t/2) \kappa$ we recall from
Lemma \ref
{lemEK} that $|\E[\kappa_{t+1} \mid\cF_t] - \kappa^*| \leq\frac
{1}{4}$ and obtain that
%
%
\begin{eqnarray}
\label{eq-kappa-10-a}
\P( | \kappa_{t+1} - \kappa^* | > \kappa^{2/3} \mid\cF_t)
&\leq&
2\exp\bigl(-\tfrac12\bigl(\kappa^{2/3}-\tfrac14\bigr)^2 / \kappa\bigr)\nonumber\\
&=&
2\exp\bigl(-\tfrac12\kappa^{1/3} + O(\kappa^{-1/3})\bigr) \\
&<& \kappa
^{-100},\nonumber
\end{eqnarray}
where the last inequality holds for any sufficiently large $\kappa$,
thus establishing~(\ref{eqK-evol-conc}).

To obtain (\ref{eqG-evol-conc}), recall from (\ref
{eqexp-lip}) that $-1 \leq e^{-(w_i + w_j)s} - e^{-w_i s} -
e^{-w_j s} \leq0$, implying that the random variable $F_{t+1} (\kappa
^*)$ is 1-Lipschitz w.r.t. the aforementioned $\kappa$-dimensional
product space. Furthermore, $\E[ F_{t+1} (\kappa^*) \mid\cF_t ]
\geq
[ G_t(1) + \epsilon_t^{13/\epsilon_t} ] \kappa^* -2$
due to Lemma~\ref{lemE-increment}, and by the same argument as before
we conclude from Hoeffding's inequality that
\begin{eqnarray*}
&&\P\bigl(
F_{t+1}(\kappa^*) < [
G_t(1) + \epsilon_t^{13/\epsilon_t}
]\kappa^*
- \kappa^{2/3} \mid\cF_t
\bigr)\\
&&\qquad\leq\exp\bigl(-\tfrac12\kappa^{1/3} + O(\kappa^{-1/3})\bigr) <
\kappa
^{-100}.
\end{eqnarray*}
Rewriting this inequality in terms of $G_{t+1}$, with probability at
least $1 - \kappa^{-100}$ we have
\begin{eqnarray*}
G_{t+1}\biggl( \frac{\kappa^*}{\kappa_{t+1}}\biggr)
&\geq&[
G_t(1) + \epsilon_t^{13/\epsilon_t}
] \frac{\kappa^*}{\kappa_{t+1}} - \frac{\kappa^{2/3}}{\kappa
_{t+1}}\\
&\geq&
G_t(1) + \epsilon_t^{13/\epsilon_t}
- \frac{2|\kappa_{t+1}-\kappa^*|+\kappa^{2/3}}{\kappa_{t+1}},
\end{eqnarray*}
where we used that $[G_t(1) + \epsilon_t^{13/\epsilon_t}](\kappa
^*-\kappa_{t+1}) \geq-(G_t(0) + 1)|\kappa_{t+1}-\kappa^*| =
-2|\kappa
_{t+1}-\kappa^*|$ due to $G_t(s)$ being decreasing in $s$. Moreover,
since $G_{t+1}$ is 1-Lipschitz as was shown in Claim~\ref{clmG-lip},
in this event we have
\[
G_{t+1}(1)
\geq
G_{t+1} \biggl(\frac{\kappa^*}{\kappa_{t+1}}\biggr) - \biggl|1-
\frac
{\kappa^*}{\kappa_{t+1}} \biggr|
\geq
G_t(1) + \epsilon_t^{13/\epsilon_t}
- \frac{3|\kappa_{t+1}-\kappa^*|+\kappa^{2/3}}{\kappa_{t+1}}.
\]
Finally, recalling from (\ref{eq-kappa-10-a}) that $|\kappa
_{t+1}-\kappa^*| \leq\kappa^{2/3}$ except with a probability of at
most $\kappa^{-100}$, we can conclude that with probability at least
$1-2\kappa^{-100}$
\begin{eqnarray*}
G_{t+1}(1)
&\geq&
G_{t+1} \biggl(\frac{\kappa^*}{\kappa_{t+1}}\biggr) - \biggl|1-
\frac
{\kappa^*}{\kappa_{t+1}} \biggr|
\geq
G_t(1) + \epsilon_t^{13/\epsilon_t}
- 4 \frac{\kappa^{2/3}}{\kappa_{t+1}} \\
&\geq&
G_t(1) + \epsilon_t^{13/\epsilon_t}
- 8 \kappa^{-1/3},
\end{eqnarray*}
where the last inequality used the fact that $\kappa_{t+1} \geq\kappa
/2$ by definition of the coalescence process (since the merging pairs
of clusters are always pairwise-disjoint).
This yields (\ref{eqG-evol-conc}) and therefore completes the proof of
the proposition.

\section*{Acknowledgments}

We thank Yuval Peres and Dahlia Malkhi for suggesting the problem and
for useful discussions. The starting point of our work is attributed to
the analytic approximation framework of Oded Schramm, and E. Lubetzky is
indebted to Oded for enlightening and fruitful discussions on his
approach.

This work was initiated while P.-S. Loh was an intern at the Theory
Group of Microsoft Research, and he thanks the Theory Group for its
hospitality.

%
%

\printaddresses

\end{document}